\documentclass{siamonline190516}
\usepackage{a4,latexsym,exscale,theorem,epsfig}
\usepackage{amssymb,psfrag,epsf,amsmath,verbatim,bbm,float}
\usepackage{listings}
\usepackage{color}
\usepackage{mathbbol}
\usepackage{enumerate}
\newtheorem{remark}[theorem]{Remark}

\usepackage{geometry}
\usepackage{accents}
\usepackage{algpseudocode}
\usepackage{todonotes}
\usepackage{caption}
\begin{document}
\newcommand {\eps} {\varepsilon}
\newcommand {\Z} {\mathbbm{Z}}
\newcommand {\R} {\mathbbm{R}}
\newcommand {\N} {\mathbbm{N}}
\newcommand {\C} {\mathbbm{C}}
\newcommand {\I} {\mathbbm{I}}

\newcommand {\PP} {\mathbbm{P}}
\newcommand {\ang} {\measuredangle}
\newcommand {\e} {{\rm{e}}}
\newcommand {\rank} {{\rm{rank}}}
\newcommand {\Span} {{\mathrm{span}}}
\newcommand {\card} {{\rm{card}}}
\newcommand {\ED} {\mathrm{ED}}
\newcommand {\cA} {\mathcal{A}}
\newcommand {\cO} {\mathcal{O}}
\newcommand {\cF} {\mathcal{F}}
\newcommand {\cC} {\mathcal{C}}
\newcommand {\cN} {\mathcal{N}}
\newcommand {\cV} {\mathcal{V}}
\newcommand {\cG} {\mathcal{G}}
\newcommand {\cB} {\mathcal{B}}
\newcommand {\cD} {\mathcal{D}}
\newcommand {\cP} {\mathcal{P}}
\newcommand {\cQ} {\mathcal{Q}}
\newcommand {\cW} {\mathcal{W}}
\newcommand {\cT} {\mathcal{T}}
\newcommand {\cI} {\mathcal{I}}
\newcommand {\Sn}[1] {\mathcal{S}^{#1}}
\newcommand {\range} {\mathcal{R}}
\newcommand {\kernel} {\mathcal{N}}
\newcommand{\one}{\mathbb{1}}
\renewcommand{\thefootnote}{\fnsymbol{footnote}}
\newcommand{\rle}{\rotatebox[origin=c]{-90}{$\le$}}
\newcommand{\rl}{\rotatebox[origin=c]{-90}{$<$}}
\newcommand{\rg}{\rotatebox[origin=c]{-90}{$=$}}


\title{\bf Angular values of nonautonomous linear dynamical systems: Part II -- 
Reduction theory and algorithm}

\author{Wolf-J\"urgen Beyn\footnotemark[1]\qquad
  Thorsten H\"uls\footnotemark[1]
}
\footnotetext[1]{Department of Mathematics, Bielefeld University,  
33501 Bielefeld, Germany \\
\texttt{beyn@math.uni-bielefeld.de}, \texttt{huels@math.uni-bielefeld.de}}

\maketitle


 \begin{abstract}
   This work focuses on angular values of
     nonautonomous dynamical systems which have been introduced 
     for general random and (non)autonomous dynamical systems in a previous
     publication \cite{BeFrHu20}.
   The angular value of dimension $s$  measures the maximal
   average rotation which an  $s$-dimensional subspace of the phase space
   experiences through the dynamics of a discrete-time linear system. Our main
results relate the notion of angular value to the well-known dichotomy (or
Sacker-Sell) spectrum and its associated spectral bundles. In particular,
we prove a reduction theorem which shows that instead of maximizing over all
subspaces, it suffices to maximize over so-called trace spaces which
have their basis in the spectral fibers. The reduction leads to an algorithm
for computing angular values of dimensions one and two. 
  We apply the algorithm to several systems of dimension up to $4$ and
  demonstrate its efficiency to detect the fastest rotating subspace
  even if it is not dominant under the forward dynamics.
 \end{abstract}



\section{Introduction}
\label{sec0}

The concept of angular values was introduced in
  \cite{BeFrHu20} for (non)auto\-nomous and random linear discrete-time dynamical systems.
  The purpose of this notion is to measure the average of the largest principle
  angle between two subspaces at successive times und to
  maximize this value over all subspaces of a fixed dimension.
  There is a fundamental difference between such angular values and the
  well-established theory of Lyapunov
  exponents which measure the exponential growth or decay of vectors
  in a dynamical system,
  see for example \cite[Ch.3.2]{A1998}, \cite{barreira2017}, \cite[Suppl.2]{KH95}.  
Furthermore, there are important differences to other existing notions
 which quantify rotations in dynamical systems. Let  us mention primarily
the rotation number for circle homeomorphisms and for more general maps and flows,
see e.g.\ \cite{AS1989,dMvS1993,KH95,N1971}. 
While the
rotation number measures the oriented angle between vectors
and image vectors in two-dimensional subspaces, 
angular values are based on the principal angles between subspaces
which always
lie in $[0,\frac{\pi}{2}]$.  Even the principal angle
  between one-dimensional subspaces may differ from the oriented
  angle between vectors spanning the subspaces. 
  When using principal angles as a measure one loses information about 
orientation but gains applicability to discrete-time
systems in arbitrary dimensions and to subspaces of arbitrary
dimension.
 
 While \cite{BeFrHu20} sets out the basic definitions and some general relations
  between angular values,  the construction of an algorithm as well as
  explicit formulas are restricted to the autonomous and the random case. It is the
  purpose of this article to provide deeper insight into the general
  nonautonomous case and to present a robust algorithm.

 We consider a nonautonomous linear
difference equation of the form
\begin{equation}\label{diffeq}
  u_{n+1} = A_n u_n,\quad A_n \in \R^{d,d},\quad n \in \N_0
\end{equation}
and assume throughout this paper that all matrices are 
invertible and $A_n$ as well as $A_n^{-1}$ are uniformly bounded.
As regards applications, we think of \eqref{diffeq} to arise from
the linearization of a nonlinear
(non)autonomous dynamical system along a particular trajectory.
In physical
  terms, we have in mind an object which is carried materially by a time-varying fluid
  flow and for which data about its position and orientation have been observed
  at discrete time instances. Then angular values are expected to measure
  the maximal torsional stress of the object.


The solution operator $\Phi$  corresponding to \eqref{diffeq} is defined by
\begin{equation*} \label{solop}
\Phi(n,m) = 
  \begin{cases}
   A_{n-1}\cdot\ldots\cdot A_m,& \text{ for } n > m,\\
   I, & \text{ for } n = m,\\
 A_n^{-1} \cdot \ldots \cdot A_{m-1}^{-1}, & \text{ for } n<m.
  \end{cases}
\end{equation*}
To keep this paper self-contained,
some important definitions and estimates of angles 
from \cite{BeFrHu20} are summarized in Section \ref{sec1}.
Our notion of an $s$-th angular value is based on the averages
\begin{equation}\label{av}
  \frac{1}{n}a_{k+1,k+n}(V), \quad
  a_{k+1,k+n}(V)=\sum_{j=k+1}^{k+n}\ang(\Phi(j-1,0)V,\Phi(j,0)V), \quad
  k\ge 0, n \ge 1. 
\end{equation}
Here $V$ is an element of the Grassmann manifold $\cG(s,d)$ of
$s$-dimensional subspaces of $\R^d$. By
  $\ang(U,V)\in [0,\frac{\pi}{2}]$ 
we denote the largest principal angle between
two subspaces $U,V \in \cG(s,d)$, which can be computed by a singular
value decomposition (SVD), see \cite[Ch.\ 6.4.3]{GvL2013}.
There are several possibilities to pass to a limit in \eqref{av}
and take the supremum over $V \in \cG(s,d)$, for example
\begin{align} \label{innerouter}
  \bar{\theta}_s= \limsup_{n\to \infty} \frac{1}{n} \sup_{V \in \cG(s,d)}
  a_{1,n}(V), \quad \hat{\theta}_s= \sup_{V \in \cG(s,d)}\limsup_{n\to \infty}
  \frac{1}{n} a_{1,n}(V).
\end{align}
Since the supremum over the Grassmannian occurs either
  inside or outside the $\limsup$ we call $\bar{\theta}_s$ the
$s$th inner and $\hat{\theta}_s$ the $s$th outer angular value of the
system.  
  Because of the $\limsup$ in \eqref{innerouter} we give these 
  angular values the additional attribute 'upper', and then define another
  class of 'lower' angular values where $\liminf$ is used instead.
  Finally, we consider uniform versions  of  angular values  
  by first taking the supremum of $a_{k+1,k+n}(V)$ over $k \in \N_0$ and then
  proceeding as above, see Definition \ref{defangularvalues}.
  This distinguishes between long-time dynamics starting at time zero and
  longtime dynamics starting at an arbitrary later time. Such a distinction is quite common in the Lyapunov theory when passing from
Lyapunov exponents to Bohl exponents or from Lyapunov spectra to
  dichotomy spectra (\cite{Barabanov2017}, \cite{dev10}, \cite{hp14}).
 All of the above notions of angular values can be different as
 examples in \cite[Section 3.2]{BeFrHu20} show. 
At the end of Section
\ref{sec4.3} we will reconsider these examples in view of our reduction theory.

Our numerical methods aim at computing outer angular values 
  for general nonautonomous systems. Recall that the more specific theory and
  the algorithm from \cite[Section 5, 6]{BeFrHu20} are restricted to an
  autonomous system \eqref{diffeq}, i.e.\ $A_n = A$ for $n\in\N_0$.
  In this case all first angular values mentioned above agree and can be
  computed from orthogonal bases of invariant subspaces which belong to 
eigenvalues of $A$ of the same modulus.
This reduces the numerical effort substantially to 
a series of Schur decompositions and to one-dimensional optimization.

One major goal of this article is
to develop a reduction theorem which generalizes the autonomous results  
to nonautonomous systems and to subspaces of arbitrary dimension.
We tackle this task in Section \ref{sec4}. 
It turns out that the dichotomy spectrum, also called the Sacker-Sell
spectrum \cite{ss78}, and its accompanying
spectral bundle take over the role of invariant subspaces from the
autonomous case. With every element $V \in \cG(s,d)$ we associate
a subspace $\cT_0(V)$, called the trace space, which has the same dimension
as $V$ and
which has a basis composed of vectors from the fibers; see
\eqref{tracedef}. Our main reduction result
Theorem \ref{cor1} then states that the $\limsup$ of $\frac{1}{n}a_{1,n}(V)$ from
\eqref{av} for the given space $V$ agrees with the $\limsup$ of
$\frac{1}{n}a_{1,n}(\cT_0(V))$ for the trace space.
Similar results are derived for the inner angular values ($s=1$) and
for uniform angular values ($s \ge 1$) in Sections \ref{sec4.3a} and \ref{sec4.4}.

The reduction theorem is the basis for the numerical algorithm which we propose
in \mbox{Section \ref{sec5}.}
Note that angular values are generally not achieved in the most stable
or most unstable directions of the trace spaces. Therefore,
algorithms, which  use forward iteration exclusively,
tend to fail since they follow asymptotic dynamics.

Our overall algorithm consists of the following steps:
\begin{enumerate}
\item Compute an approximation of the dichotomy spectrum.
\item Compute the corresponding spectral bundles and obtain the trace spaces.
\item Determine the supremum of \eqref{av} w.r.t.\ the trace spaces and
  for large values of $n$.
\end{enumerate}
The first two tasks are accomplished by using discrete
  versions of QR-methods from \cite{dev10} and least squares techniques from \cite{hu09}. However, by the third step our approach differs
  substantially from taking QR or SVD decompositions for an overall
  method, which is common for Lyapunov exponents; see
  \cite{de06},\cite{dev10}.
We apply this algorithm to several models  to illustrate various
aspects. We begin with a comparison between explicitly known angular values
and the output of our algorithm  for  cases $s=1,2$. 
 For the classical H\'{e}non system we illustrate the
    somewhat slow convergence of the ergodic averages \eqref{av} as $n \to \infty$
    and we discuss their dependence on the initial time $k$.
    Further, we give a geometric
    interpretation of  angular values as a measure of
    the maximal average angle between 
    successive tangent  lines to the stable  and the unstable fiber bundle within the
    attractor. In this case the stable fiber  rotates
    faster than the unstable one while the latter
    is dominant under the forward dynamics.
    This  demonstrates that our reduction procedure
    is essential in detecting the fastest rotating subspace.\\
    We continue the geometric interpretation of angular values for a
$3$-dimensional extension of the H\'{e}non system and a $4$-dimensional system
of coupled oscillators. For the coupled oscillator model we observe
that angular values are insensitive to the breakdown of
an invariant torus at increasing coupling values. The torus breakdown is known to be related to the ratio of Lyapunov exponents
towards vs.\ inside the torus (\cite{dlr91}, \cite{dl95}, \cite{r2000}).
Rather, angular values reflect the rotation of Floquet spaces associated
with the periodic orbit approached by the trajectory. This behavior occurs
regardless of whether the periodic orbit is imbedded into an invariant torus or not.


\section{Basic definitions and properties}\label{sec1}
To keep the article self-contained,
we summarize in this section some important notions, definitions and results
from \cite{BeFrHu20}. 

\subsection{Angles and subspaces} \label{sec1.0}
Let us begin with a useful characterization of the angle between two
subspaces $V$ and $W$ of $\R^d$, both having the same dimension
$s$.
Principal angles between these subspaces can be computed from the
singular values of $V_B^\top W_B$, where the columns of $V_B$ and
$W_B\in \R^{d,s}$
form orthonormal bases of $V$ and $W$, respectively, see
\cite[Ch.6.4.3]{GvL2013}. The smallest singular value is the cosine of
the largest principal angle which we denote by $\ang(V,W)$. 
We further use the notion 
\begin{align*}
  \ang(v,w) = \ang(\mathrm{span}(v),\mathrm{span}(w)), \quad v,w \in \R^d,
  v,w \neq 0
\end{align*}
in case the subspaces are one-dimensional.

The following characterization of
$\ang(V,W)$ turns out to be essential for the analysis of angular
values. It provides an alternative to the standard
  definition in \cite[Ch.6.4.3]{GvL2013}, see \cite[Supplementary
    materials I]{BeFrHu20} for a proof.
  \begin{proposition}\label{Lemma2}
  Let $V, W\subseteq \R^d$ be two $s$-dimensional subspaces. 
  Then the following relation holds
  \begin{equation*}\label{A1}
    \ang(V,W) =
    \max_{\substack{v\in V \\ v\neq 0} }\min_{\substack{w\in W \\ w \neq 0}} \ang(v,w)
  = \arccos\big(\min_{\substack{v\in V\\\|v\|=1}} \max_{\substack{
      w\in W\\\|w\|=1}} v^{\top} w\big).
  \end{equation*}
\end{proposition}
The next proposition summarizes 
some well-known properties of the Grassmannian
\begin{equation*} \label{eq1.4}
  \mathcal{G}(s,d) = \{ V \subseteq \R^d \; \text{is a subspace of dimension}
  \; s \},
\end{equation*}
see \cite[Ch.6.4.3]{GvL2013}, \cite{js96}. Throughout the paper $\|\cdot\|$
denotes the Euclidean norm of vectors as well as the associated spectral
norm of matrices.
\begin{proposition} \label{prop2g}
  The Grassmannian $\mathcal{G}(s,d)$ is a compact smooth manifold of
  dimension $s(d-s)$ and a metric space with respect to
  \begin{equation*} 
    d(V,W) = \| P_V - P_W \|,
  \end{equation*}
  where $P_V,P_W$ are the orthogonal projections onto $V$ and $W$, respectively.
 Moreover, the formula
\begin{align*}
  d(V,W) = \sin (\ang(V,W)), \quad V,W \in \mathcal{G}(s,d) 
\end{align*}
holds and  $\ang(V,W)$ defines an equivalent metric on 
$\mathcal{G}(s,d)$ satisfying 
$$
\frac{2}{\pi} \ang(V,W) \le d(V,W)
  \le \ang(V,W).
$$
\end{proposition}
The following lemma from \cite[Lemma 2.6]{BeFrHu20} is our main tool to
estimate angles of vectors and subspaces in terms of the Euclidean norm.
\begin{lemma} \label{lem0}\ \\[-5mm]
  \begin{enumerate}[i)]
  \item For any two vectors $v,w \in \R^d$ with $\|v\| < \|w\|$ the following holds
    \begin{equation*} \label{eq1.6}
      \begin{aligned}
        \tan^2\ang(v+w,w)& \le \frac{\|v\|^2}{\|w\|^2 - \|v\|^2}.
      \end{aligned}
    \end{equation*}
  \item
    Let $V \in \mathcal{G}(s,d)$ and $P \in \R^{d,d}$ be such that
    for some $0 \le q <1$
    \begin{equation*} \label{eq1.7}
      \| (I-P)v\| \le q \| Pv \| \quad \forall v \in V.
    \end{equation*}
    Then $\dim(V)=\dim(PV)$ and the following estimate holds
    \begin{equation*} \label{eq1.8}
      \ang(V,PV) \le \frac{q}{(1-q^2)^{1/2}}.
    \end{equation*}
  \end{enumerate}
\end{lemma}
 
Finally, we state a linear algebra result which will provide the basic
reduction step in Theorem \ref{Th2}. By $\range(P)$ and $\kernel(P)$ we denote
the range and kernel of a matrix $P$, respectively.
\begin{lemma} \label{linalg}
  Let $V \in \mathcal{G}(s,d)$ and let $P$ be a projector in $\R^d$. Furthermore,
  let $Q$ be any projector defined in $\R^d$ and with range $V\cap \range(P)$.
  Then the linear map
 \begin{align} \label{Lmap}
     L=I-P+Q:  V  \to  (I-P)V \oplus (V \cap \range(P)) 
      \end{align}
  is a bijection and there exists a constant $\rho>0$ such that
  \begin{align} \label{estproj}
    \|P(I-Q)v \| \le \rho \|(I-P)(I-Q)v \| \quad \forall v \in V.
  \end{align}
\end{lemma}
\begin{proof}
  Note that the sum in \eqref{Lmap} is direct since $(I-P)V \subseteq \kernel(P)$.
   For the same reason, if $Lv=0$ for some $v \in V$, then
  $(I-P)v=0$ and $Qv=0$ holds. This shows $v \in V \cap \range(P)=\range(Q)$ 
  and $v = Qv=0$. Thus $L$ is one to one.
  To show that $L$ is onto, take  any $u \in (I-P)V$ and $w \in V
  \cap\range(P)$. Then we have 
  $u=(I-P)v$ for some $v \in V$, and defining $\tilde{v}=(I-Q)v+w\in V$, we
  obtain
  $L\tilde{v} = (I-P)(v+w - Qv)+Qw=(I-P)v+w=u+w$.
  To show the estimate \eqref{estproj} note that
   $0=(I-P)(I-Q)v= (I-P)v$ for some $v \in V$ implies $v \in V \cap \range(P)=\range(Q)$
  and thus $(I-Q)v=0=P(I-Q)v$. Then we obtain \eqref{estproj} from
  the elementary fact
  that two linear maps $A,B\colon V \to \R^d$ satisfy
  $\kernel(B) \subseteq \kernel(A)$ if and only if there exists a constant
  $C>0$ such that $\|Av\|\le C \|Bv\|$ for all $v \in V$.
\end{proof}
\begin{remark} \label{relbounded}
  The last step of the proof is a special case of a result from functional analysis:
  Let $A,B\colon X \to Y$ be linear bounded operators between Banach spaces
  $X$ and $Y$,
  then $A$ is called relatively bounded by $B$ if there
  exists a constant $C>0$ such that $\|Ax\| \le C \|Bx\|$ for all $x \in X$;
  see \cite[Ch.3.7]{Edmunds02}. 
The smallest constant of this type is
\begin{equation} \label{ABbounded}
  \rho(A,B) = \inf\{C>0: \|Ax\|\le C \|Bx\| \quad \forall x\in X \}
  =\sup_{Bx \neq 0} \frac{ \|Ax\|} {\|Bx\|}.
\end{equation}
If  $B$ is  Fredholm one can show
that $A$ is relatively bounded by $B$ iff
$\kernel(B) \subseteq \kernel(A)$.
\end{remark}
\subsection{Definition of angular values}
\label{sec1.1}
  The angular values defined in \cite[Section 3.1]{BeFrHu20} always
  aim at finding the subspace $V$ which maximizes the average
  angle  in \eqref{av}. However, there are several possibilities to
  let time go to infinity. We will use the attributes 'upper' and 'lower'
  to distinguish between $\limsup$ and $\liminf$, and the attributes
  'outer' and 'inner' to distinguish between limits taken outside or
  inside the supremum over the Grassmannian, cf.\ \eqref{innerouter}.
  This motivates the first part of the following definition.

\begin{definition} \label{defangularvalues}
  Let the  nonautonomous system \eqref{diffeq} be given. For $s\in
  \{1,\ldots,d\}$ define the quantities 
  \begin{equation} \label{defsums}
    a_{m,n}(V) = \sum_{j=m}^{n} \ang(\Phi(j-1,0)V,\Phi(j,0)V) \quad 
   m,n \in \N,\; V\in \mathcal{G}(s,d).
    \end{equation}
  \begin{enumerate}[i)]
  \item
    The \textbf{upper} resp.\ \textbf{lower} $s$th \textbf{inner angular value} is defined by
    \begin{equation}\label{dinner}
\begin{aligned}
  \bar \theta_s =  \limsup_{n\to\infty} \frac{1}{n} \sup_{ V \in \mathcal{G}(s,d)}
   a_{1,n}(V),\quad 
  \underaccent{\bar}\theta_s = \liminf_{n\to\infty} \frac{1}{n}
  \sup_{ V \in \mathcal{G}(s,d)}
   a_{1,n}(V).
\end{aligned}
    \end{equation}
  \item
    The \textbf{upper} resp.\ \textbf{lower} $s$th \textbf{outer angular value} is defined by
    \begin{equation*}\label{douter}
\begin{aligned}
  \hat \theta_{s} = \sup_{V \in \mathcal{G}(s,d)}
  \limsup_{n\to\infty} \frac{1}{n}  a_{1,n}(V), \quad
    \underaccent{\hat}\theta_s =\sup_{V \in \mathcal{G}(s,d)}
  \liminf_{n\to\infty} \frac{1}{n}  a_{1,n}(V).
\end{aligned}
    \end{equation*}
    \item
    The \textbf{upper} resp.\ \textbf{lower} $s$th \textbf{uniform inner
      angular value} is defined by 
\begin{equation}\label{duniinner}
\begin{aligned}
  \bar \theta_{[s]} = &
  \lim_{n\to\infty} \frac{1}{n} \sup_{V \in \mathcal{G}(s,d)}
  \sup_{k\in\N_0}a_{k+1,k+n}(V),\\ 
  \underaccent{\bar}\theta_{[s]} = & 
  \liminf_{n\to\infty} \frac{1}{n}\sup_{V \in \mathcal{G}(s,d)} \inf_{k\in\N_0} a_{k+1,k+n}(V).
\end{aligned}
\end{equation}
  \item
    The \textbf{upper} resp.\ \textbf{lower} $s$th \textbf{uniform outer angular value} is defined by
\begin{equation*}\label{duniouter}
\begin{aligned}
  \hat \theta_{[s]} = & \sup_{V \in \mathcal{G}(s,d)}
  \lim_{n\to\infty} \frac{1}{n} \sup_{k\in\N_0}a_{k+1,k+n}(V), \\
  \underaccent{\hat}\theta_{[s]} = & \sup_{V \in \mathcal{G}(s,d)}
  \lim_{n\to\infty} \frac{1}{n} \inf_{k\in\N_0} a_{k+1,k+n}(V).
\end{aligned}
\end{equation*}
\end{enumerate}
\end{definition}
The uniform angular values defined in (iii) and (iv)
  distinguish between longtime dynamics starting at time zero and
  longtime dynamics starting at an arbitrary later time.
  In the stability theory for nonautonomous systems, such a distinction
  is quite common when passing 
  from  Lyapunov exponents to Bohl exponents, see \cite{Barabanov2017},
  \cite[Ch.III.4]{dk74}.
  One can show that the limits occurring in Definition \ref{defangularvalues}
  exist and that the various angular values are related by the diagram
  below (\cite[Lemma 3.3]{BeFrHu20})
  \begin{equation}\label{eq2.4}
    \begin{matrix}
    \underaccent{\hat} \theta_{[s]} & \le & \underaccent{\hat}
    \theta_s & \le & \hat \theta_s & \le & \hat \theta_{[s]}\\
    \rle& & \rle && \rle && \rle\\
    \underaccent{\bar} \theta_{[s]} & \le & \underaccent{\bar}
    \theta_s & \le & \bar \theta_s & \le & \bar \theta_{[s]}\\
    \end{matrix}
  \end{equation}
    Finally, we note that all inequalities in this diagram can be strict,
as suitable examples in \cite[Section 3.2]{BeFrHu20}  show. However,
    our numerical computations in Section \ref{sec5} suggest that such a
    distinction of angular values is rather exceptional.
  
  
\section{The dichotomy spectrum and reduction theorems}
\label{sec4}
We discuss in this section relations between angular values and the
dichotomy spectrum. This particularly results in a computational
approach for angular values. We start with a brief introduction,
cf.\ \cite{h17}, of the dichotomy spectrum which is also called 
the Sacker-Sell spectrum.

\subsection{Dichotomy spectrum and trace spaces}\label{SackerSell}
The dichotomy spectrum, see \cite{ss78} is based on the
notion of an exponential dichotomy, cf.\ \cite{he81, ak01, k94, co78,
  dk74, p30}. In the following we recall its general definition
for a
discrete interval $\I\subset \Z$ which is unbounded
above, and for a linear system  
\eqref{diffeq}  which is bounded invertible, i.e.\  there exists a $C>0$ such
that $\|A_n\|,\|A_n^{-1}\| \le C$ for all $n\in\I$.
\begin{definition}\label{edDef}
The difference equation \eqref{diffeq} has an \textbf{exponential
  dichotomy} (\textbf{ED} for short)
on $\I$, if there exist constants $K>0,\alpha_s,\alpha_u \in(0,1)$ and families of
projectors $P_n^s$, $P_n^u := I-P_n^s$, $n\in\I$ such that 
\begin{itemize}
\item [(i)] $P_{n}^{s,u} \Phi(n,m) = \Phi(n,m)P_m^{s,u}$ for all
  $n,m\in\I$.
\item[(ii)] For $n, m\in\I$,  $n\ge m$ the following estimates hold:
$$
\|\Phi(n,m) P_m^s\| \le K\alpha_s^{n-m},\quad 
\|\Phi(m,n) P_n^u\| \le K \alpha_u^{n-m}.
$$
\end{itemize}
The tuple $(K,\alpha_{s,u}, P_\I^{s,u}=(P_n^{s,u})_{n \in \I})$ is called the dichotomy data. 
\end{definition}

The dichotomy spectrum is constructed, using the scaled equation
\begin{equation}\label{scale}
u_{n+1} = \frac 1 \gamma A_n u_n,\quad n\in\I,
\end{equation}
which has the solution operator $\Phi_\gamma(n,m) =
   \gamma^{m-n} \Phi(n,m)$.
Spectrum and resolvent set are defined as follows: 
$$
\Sigma_\ED := \{\gamma > 0 : \eqref{scale} \text{ has no ED on } \I\},\qquad
 R_\ED:= \R^{> 0}\setminus \Sigma_\ED. 
$$
The Spectral Theorem \cite[Theorem 3.4]{as01} provides the
decomposition
$\Sigma_\ED = \bigcup_{i=1}^\ell \cI_i$ of the dichotomy spectrum into
$\ell \le d$ intervals
$$
\cI_i = [\sigma_i^-, \sigma_i^+], \ i = 1,\dots,\ell,\quad \text{where}\quad
0 < \sigma_{\ell}^- \le \sigma_{\ell}^+ < \dots < \sigma_1^- \le \sigma_1^+.
$$
The intervals $\cI_i$, $i=1,\dots,\ell$ are called spectral
intervals. 
Correspondingly, the resolvent set
is $R_\ED =
\bigcup_{i=1}^{\ell+1} R_i$ with resolvent intervals 
\begin{equation*} \label{resolventset}
R_1 = (\sigma_1^+,\infty),\quad R_i = (\sigma_i^+,\sigma_{i-1}^-),\
i=2,\dots,\ell+1 \text{ with }\sigma_{\ell+1}^+ = 0,
\end{equation*}
see Figure \ref{F01}.
In case $\sigma_i^- = \sigma_{i}^+$ for an
$i\in \{1,\dots,\ell\}$ the spectral interval $\mathcal{I}_i$
is an isolated point. 
\begin{figure}[hbt]
\begin{center}
\includegraphics[width=0.75\textwidth]{ 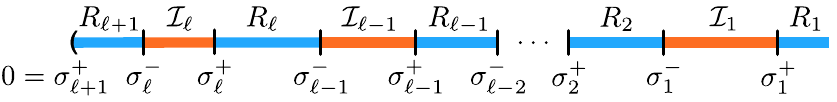}   
\end{center}
\caption{\label{F01}Illustration of spectral intervals (orange)  and of
  resolvent intervals (blue).}  
\end{figure}
For $\gamma \in R_i=(\sigma_i^+,\sigma_{i-1}^-)$
  with $i\in \{1,\dots,\ell+1\}$ the system \eqref{scale} has an ED
    with dichotomy data
\begin{equation}\label{const_alpha}
  \Big(K\, ,\alpha_s(\gamma) = \frac {\sigma_i^+}\gamma\, ,
  \alpha_u(\gamma) = \frac \gamma {\sigma_{i-1}^-}\, ,
  (P_{n,i}^s,P_{n,i}^u=I-P_{n,i}^s)_{n \in \I}   \Big) .
 \end{equation}
The projectors $P_{n,i}^{s,u}$ commute with $\Phi_{\gamma}(n,m)=\gamma^{m-n}\Phi(n,m)$ as in Definition \ref{edDef} (i) and satisfy
for $n\ge m$ the ED-estimates 
\begin{align*}
\|\Phi(n,m) P_{m,i}^s\|&= \gamma^{n-m} \|\Phi_\gamma(n,m) P_{m,i}^s\| \le K
  \gamma^{n-m}\left(\frac
  {\sigma_i^+}\gamma\right)^{n-m}=K(\sigma_i^+)^{n-m},\\
\|\Phi(m,n) P_{n,i}^u\|&= \gamma^{m-n} \|\Phi_\gamma(m,n)P_{n,i}^u\| \le K
  \gamma^{m-n}\left(\frac
  \gamma{\sigma_{i-1}^-}\right)^{n-m}=K\left(\sigma^-_{i-1}\right)^{m-n}.
\end{align*}
Note that these estimates do not depend on the particular choice of
$\gamma \in R_i$. Further,
the projectors are uniquely determined and   their ranges form
  a flag of subspaces for $n\in\I$, i.e.
\begin{equation} \label{projhierarchy}
  \begin{aligned}
\{0\}& = \range(P_{n,\ell+1}^s) \subseteq
        \range(P_{n,\ell}^s)\subseteq & \dots \subseteq \range(P_{n,1}^s) =
        \R^d,\\
\R^d& = \range(P_{n,\ell+1}^u) \supseteq
        \range(P_{n,\ell}^u)\supseteq &\dots \supseteq \range(P_{n,1}^u) =
        \{0\}.
  \end{aligned}
  \end{equation}

Spectral bundles that correspond to eigenspaces in autonomous systems
are defined as follows (see Figure \ref{F02}): 
\begin{equation}\label{specbun}
\begin{aligned}
\cW^i_n &:= \range (P_{n,i}^s) \cap \range (P_{n,i+1}^u),\quad i =
1,\dots, \ell.
\end{aligned}
\end{equation}
Note that the dimensions of these spectral
bundles $d_i:= \dim(\cW_n^i)$ for $i=1,\dots,\ell$ do not depend on
$n\in\I$.
Alternatively, we may write the ranges of dichotomy projectors in
terms of spectral bundles:
\begin{equation}\label{dichbundle}
  \range(P_{n,i}^u) = \bigoplus_{j=1}^{i-1} \cW_n^j,\quad
  \range(P_{n,i}^s) = \bigoplus_{j=i}^\ell \cW_n^j,
  \quad i = 1,\dots,\ell+1.
  \end{equation}
The fiber projector $\cP_{n,i}, i=1,\ldots,\ell$  onto $\cW_n^i$ along
$\bigoplus_{\nu=1,\nu\neq i}^{\ell} \cW_n^{\nu}$ is given by
\begin{equation} \label{fiberproj}
  \cP_{n,i}=P_{n,i}^s P_{n,i+1}^u=P_{n,i+1}^u P_{n,i}^s=P_{n,i}^s- P_{n,i+1}^s
  =P_{n,i+1}^u - P_{n,i}^u.
\end{equation}
Spectral bundles satisfy for $i=1,\dots,\ell$ and 
$n,m\in\I$ the invariance condition
\begin{equation}\label{invar}
\Phi(n,m)\cW_m^i = \cW_n^i.
\end{equation}

\begin{figure}[hbt]
\begin{center}
\includegraphics[width=0.8\textwidth]{ 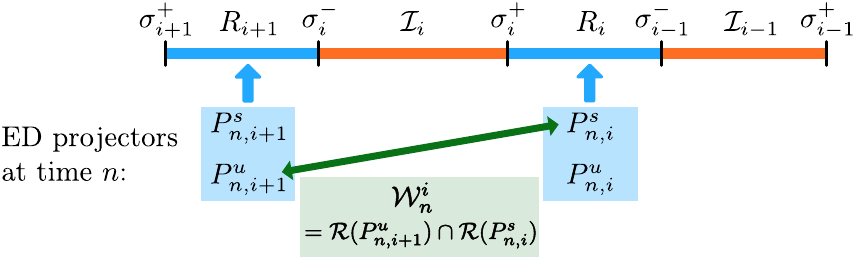}   
\end{center}
\caption{\label{F02}Construction of spectral bundles.}  
\end{figure}
For later purposes, we define the following subset of the Grassmannian
$\cG(s,d)$. 
\begin{definition}\label{deftrace}
Every element $V\in \cG(s,d)$ of the form
\begin{equation*} \label{tracespaceX}
  V = \bigoplus_{i=1}^{\ell}W_i: W_i \subseteq \cW_k^i \;
  \text{(subspace)}\; i=1,\ldots,\ell,\;
\sum_{i=1}^{\ell}\dim W_i=s
\end{equation*}
is called a \textbf{trace space} at time $k$. 
The set of all trace spaces is denoted by $\cD_k(s,d)$.
\end{definition}
Trace spaces depend on the spectral bundle in
  \eqref{specbun} and determine all angular values. Therefore, they
  form an essential part of the Grassmannian and we will study them in
  the following subsections.

\subsection{Outer angular values and spectral bundles}
\label{sec4.3}
In the remaining part of this section, we work in the setup $\I =
\N_0$. 
As a first step,  we consider a system with an ED and study the
  forward dynamics $\Phi(j,k)$ of a subspace $V \in \cG(s,d)$ starting at an arbitrary time
  $k \in \N$. Decomposing  $V=(\range(P^s_k)\cap V)\oplus \tilde{V} $ where
  $P_k^u \tilde{V}=P_k^uV$, we see that the stable part $\range(P^s_k)\cap V$  moves to $\Phi(j,k)(\range(P^s_k)\cap V)=\range(P^s_j)\cap \Phi(j,k)V$ because of invariance, while  
  the remaining part $\tilde{V}$ is driven towards the image
  $\Phi(j,k)P^u_k V=P^u_j \Phi(j,k)V$.
  Therefore, we expect the image of $V$ to approach $(\range(P^s_j)\cap \Phi(j,k)V)\oplus
  \Phi(j,k)P^u_k V$ as illustrated in Figure \ref{Th23}. This is made precise
  in the following theorem.

\begin{figure}[hbt]
    \begin{center}
      \includegraphics[width=0.95\textwidth]{ 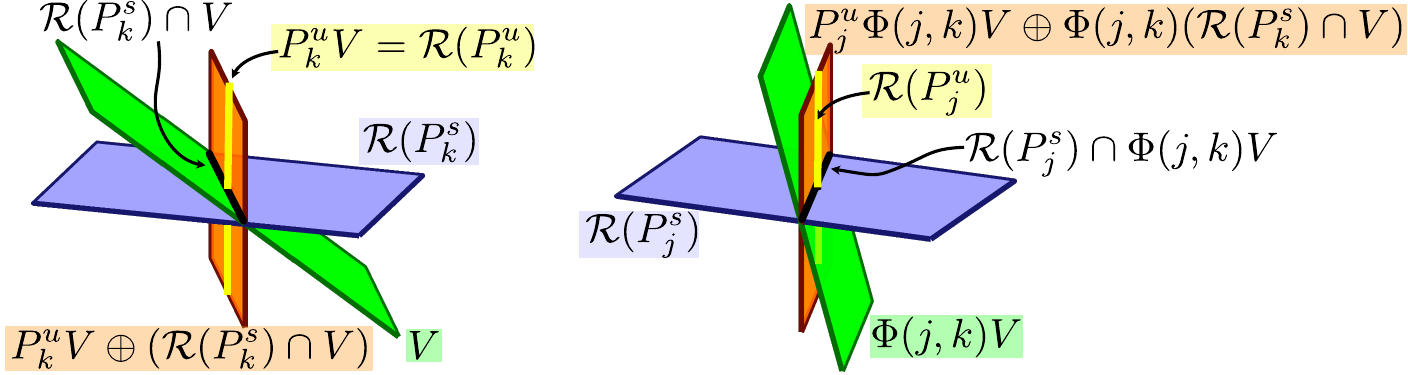}   
    \end{center}
\caption{\label{Th23} A
subspace $V$ (green) at time $k$ (left) is driven by the dynamics at time $j>k$
(right) towards the direct sum (orange plane) of its unstable
projection (yellow) and the intersection
with the stable subspace (black).}
\end{figure}

\begin{theorem} \label{Th1.1}
  Let the system \eqref{diffeq} have an exponential dichotomy on $\N_0$ with data
  $(K,\alpha_{s,u},P^{s,u}_{\N_0})$. For 
  $k \in \N_0$ and  $V \in \cG(s,d)$ let 
 \begin{equation} \label{Qkdef}
    Q_k^s:\R^d \to \range(P^s_k)\cap V
  \end{equation}
  denote the orthogonal projector onto $\range(P^s_k)\cap V$.
  Then the quantity (recall \eqref{ABbounded}) 
    \begin{align*}
   \rho^s_k(V)=  \inf\{C>0:
    \|P^s_k(I-Q_k^s)v  \| \le C \|P^u_k(I-Q_k^s) v \| \; \forall v \in V \} < \infty
    \end{align*}
is finite and 
  there exists an index $j^s_k=j^s_k(V)$ such that
  \begin{equation} \label{indexlarge}
    K^2 (\alpha_s \alpha_u)^{j^s_k} \rho^s_k(V)  \le \frac{1}{2}.
  \end{equation}
  For all $j \ge k+ j^s_k$ the following estimate holds
    \begin{equation} \label{estanguns}
    \begin{aligned}
     \ang( \Phi(j,k)V,\Phi(j,k)(P_k^u V \oplus Q_k^s V))
    & \le \frac{2}{\sqrt{3}}K^2
    (\alpha_s \alpha_u)^{j-k} \rho^s_k(V).
    \end{aligned}
  \end{equation}
  \end{theorem}
\begin{proof}
  Applying Lemma \ref{linalg} to $P=P^s_k$ and $Q=Q^s_k$ shows
    that the quantity $\rho^s_k(V)$ is finite.
   Since $\alpha_s \alpha_u<1$ there exists an index 
   $j^s_k$ satisfying \eqref{indexlarge}.
       Our goal is to apply Lemma \ref{lem0} (ii) for $j \ge k$
    to the $s$-dimensional subspace $\tilde{V}=\Phi(j,k)V$  and the matrix
    $\tilde{P}= P_j^u + P_j^s\Phi(j,k)Q_k^s \Phi(k,j)$. First note that
    Lemma \ref{linalg} and the properties of the solution operator imply
    \begin{align*}
      \Phi(j,k)(P_k^uV \oplus Q_k^sV)& 
      = \Phi(j,k)(P_k^u+Q_k^s) \Phi(k,j)\Phi(j,k)V\\
      &=(P_j^u+ \Phi(j,k)Q_k^s\Phi(k,j))\Phi(j,k)V 
      =\tilde{P}\tilde{V}.
    \end{align*}
   The exponential dichotomy yields for all $v \in V$ and $j \ge k$
  \begin{align*}
    \|P_k^u v \| & = \|\Phi(k,j)P_j^u \Phi(j,k)(P_k^u v + Q_k^s v) \| \\
    & \le K \alpha_u^{j-k} \|\Phi(j,k)(P_k^u v + Q_k^s v) \|= K \alpha_u^{j-k}
    \|\tilde{P}\Phi(j,k) v \|, \\
    \|(I- \tilde{P})\Phi(j,k)v\| & = \| \Phi(j,k)v - (\Phi(j,k)P_k^uv +
    \Phi(j,k) P_k^s Q_k^s v) \| \\
    & = \| \Phi(j,k)P_k^s(I- Q_k^s) v\| \le K \alpha_s^{j-k}\|P_k^s(I-Q_k^s)v\|.
  \end{align*}
  Combining these estimates we obtain
  \begin{equation} \label{powerest}
  \begin{aligned}
    \|(I- \tilde{P})\Phi(j,k)v \| & \le K \alpha_s^{j-k} \rho^s_k(V)
    \| P_k^u(I-Q_k^s)v \|
    =K \alpha_s^{j-k} \rho^s_k(V) \| P_k^u v \|\\
    &\le K^2 (\alpha_s \alpha_u)^{j-k}
    \rho^s_k(V) \|\tilde{P} \Phi(j,k)v \|.
  \end{aligned}
  \end{equation}
   By condition \eqref{indexlarge} we can apply
   Lemma \ref{lem0} (ii) with
   $q=K^2 (\alpha_s \alpha_u)^{j-k} \rho^s_k(V) \le \frac{1}{2}$ for
   $j \ge k+ j^s_k$ and obtain
  \begin{align*} \label{estVtoPV}
    \ang(\Phi(j,k)V, \Phi(j,k)(P_k^uV\oplus Q_k^s V)) \le
    \frac{2K^2}{\sqrt{3}}(\alpha_s 
    \alpha_u)^{j-k} \rho^s_k(V).
  \end{align*}
   This proves the estimate \eqref{estanguns}.
  \end{proof}

In a similar way, we can reverse time and analyze the
  backwards dynamics of $V$. For this purpose, let
 \begin{equation} \label{Qkudef}
    Q_k^u: \R^d \to \range(P^u_k)\cap V
  \end{equation}
  denote the orthogonal projector onto $\range(P^u_k)\cap V$.
  Then 
    \begin{align*}
   \rho^u_k(V)=  \inf\{ C>0:
    \|P^u_k(I-Q_k^u)v  \| \le C \|P^s_k(I-Q_k^u) v \| \; \forall v \in
      V\} < \infty 
    \end{align*}
holds and 
  there exists an index $j^u_k=j^u_k(V)$ such that
  \begin{equation} \label{indexularge}
    K^2 (\alpha_s \alpha_u)^{j^u_k} \rho^u_k(V)  \le \frac{1}{2}.
  \end{equation}
  Then we obtain from Lemma \ref{lem0} (ii) the following estimate for all $j \le k- j^u_k$, 
    \begin{equation} \label{estangstab}
    \begin{aligned}
     \ang( \Phi(j,k)V,\Phi(j,k)(P_k^s V \oplus Q_k^u V))
    & \le \frac{2}{\sqrt{3}}K^2
    (\alpha_s \alpha_u)^{k-j} \rho^u_k(V).
    \end{aligned}
  \end{equation}
  Theorem \ref{Th1.1} provides the building block for reducing the analysis of angular
values for general subspaces $ V \in \cG(s,d)$ to specific ones which
have a basis consisting of vectors from the spectral bundle \eqref{specbun}.
For a fixed starting time $k\in\N$ and $i=1,\ldots,\ell$ we recall the
projectors $\cP_{k,i} =  P_{k,i+1}^u P_{k,i}^s: \R^d \to \cW_k^i$ 
 from \eqref{fiberproj} and define the orthogonal projectors
(cf.\ \eqref{projhierarchy},  \eqref{dichbundle}, \eqref{Qkdef})
\begin{equation} \label{Qkidef}
   \begin{aligned}
  Q_{k,i}^s: \R^d \to \range(P_{k,i}^s)\cap V, \quad i=1,\ldots,\ell.
     \end{aligned}
\end{equation}
With each $V \in \cG(s,d)$ we associate a subspace with a
fiber basis and defined by
\begin{equation} \label{tracedef}
  \cT_k(V) = \bigoplus_{i=1}^{\ell} (\cP_{k,i} Q_{k,i}^s V)=
  \bigoplus_{i=1}^{\ell} (P_{k,i+1}^u(\range(P_{k,i}^s)\cap V) ).
\end{equation}
Below we will show $\dim \cT_k(V)=s$ and the equality
\begin{align} \label{tracerep}
  \cT_k(V) = \big(\sum_{i=1}^{\ell} \cP_{k,i} Q_{k,i}^s\big) V.
\end{align}
Therefore, $\cT_k(V)$ is a trace space at time $k$ in the sense of
Definition \ref{deftrace}. In 
fact, we have the equality 
\begin{equation} \label{tracespace}
 \cD_k(s,d)=\{\cT_k(V):V \in \cG(s,d)\},
\end{equation}
since every
 $V=\bigoplus_{i=1}^{\ell}W_i\in \cD_k(s,d)$ with $W_i \subseteq \cW_k^i$ satisfies
$\cP_{k,i}Q_{k,i}^s V = W_i$, $i=1,\ldots,\ell$.
 Hence every trace space can be found by subsequent projection as in 
  \eqref{tracerep}.

Our main reduction theorem is the following.
\begin{theorem} \label{Th2}
  Assume that the difference equation \eqref{diffeq} has the dichotomy spectrum
  $ \Sigma_\ED=\bigcup_{i=1}^{\ell}[\sigma_i^-,\sigma_i^+]$ with
  fibers $\cW_k^i,i=1\ldots,\ell$ and projectors 
  $\cP_{k,i}, i=1,\ldots,\ell$, $k \in \N_0$. Then for every $k \in \N_0$   and
  $V \in \cG(s,d)$, $s=1,\ldots,d$
  there exists an index $\overline{j}=\overline{j}(k,V)$ and a constant
  $C=C(k,V)$ such
  that for all $j \ge k+\overline{j}$ the following estimate holds
  \begin{equation} \label{angletotrace}
    \ang(\Phi(j,k)V, \Phi(j,k)\cT_k(V)) \le C(k,V) \Big( \max_{i=1,\ldots,\ell-1}
    \frac{\sigma_{i+1}^+}{\sigma_i^-}\Big)^{j-k}.
  \end{equation}
\end{theorem}
\begin{remark} Note that several of the spaces $\cP_{k,i}Q_{k,i}^s V$
    occurring in the decomposition \eqref{tracedef} may be
    trivial. The following proof will show that one can then omit  
    the corresponding quotients $\frac{\sigma_{i+1}^+}{\sigma_i^-}$
    from the maximum in \eqref{angletotrace}. Moreover, the proof will provide
    values for the index $\overline{j}(k,V)$ and the constant $C(k,V)$
     in terms of $V$ and the dichotomy data in the
      resolvent intervals.
\end{remark}
\begin{proof} The main work is to set up inductive steps
    which apply Theorem \ref{Th1.1} to $\Phi_{\gamma}$ for  values
    of $\gamma$ in successive resolvent intervals. \\
  \noindent {\bf Step1:} Let us first discuss a recursive construction that
  leads to the trace space \eqref{tracerep}.
  With every $V \in \cG(s,d)$ we associate
  subspaces $V_i \in \cG(s,d)$ and further projectors
  $\tilde{Q}_{k,i}(i=1,\ldots,\ell+1)$  defined by $V_1 = V$,
  $\tilde{Q}_{k,1} = I_d$ and then 
  for $i=1,\dots,\ell$ as follows
  \begin{equation} \label{defnewproj}
    \begin{aligned}
     & \tilde{Q}_{k,i+1}:\R^d  \to \range(P_{k,i+1}^s) \cap V_i \quad
     \text{orthogonal projector}, \\ 
   &   V_{i+1}= P_{k,i+1}^uV_i \oplus \tilde{Q}_{k,i+1} V_i.
    \end{aligned}
  \end{equation}
  Figure \ref{Th24} illustrates this recursion for two characteristic cases.
  The initial space $V=V_1$ is successively replaced by spaces
    $V_2,V_3, \ldots,V_{\ell+1}=:\cT_k(V)$ of the same dimension by working
    down the flag of subspaces in \eqref{projhierarchy}. 
    In each step the intersection $\range(P_{k,i+1}^s) \cap V_i$
    with the current stable space is kept while the remaining part is
    replaced by its current unstable projection $P_{k,i+1}^uV_i$.
\begin{figure}[hbt]
    \begin{center}
      \includegraphics[width=0.9\textwidth]{ 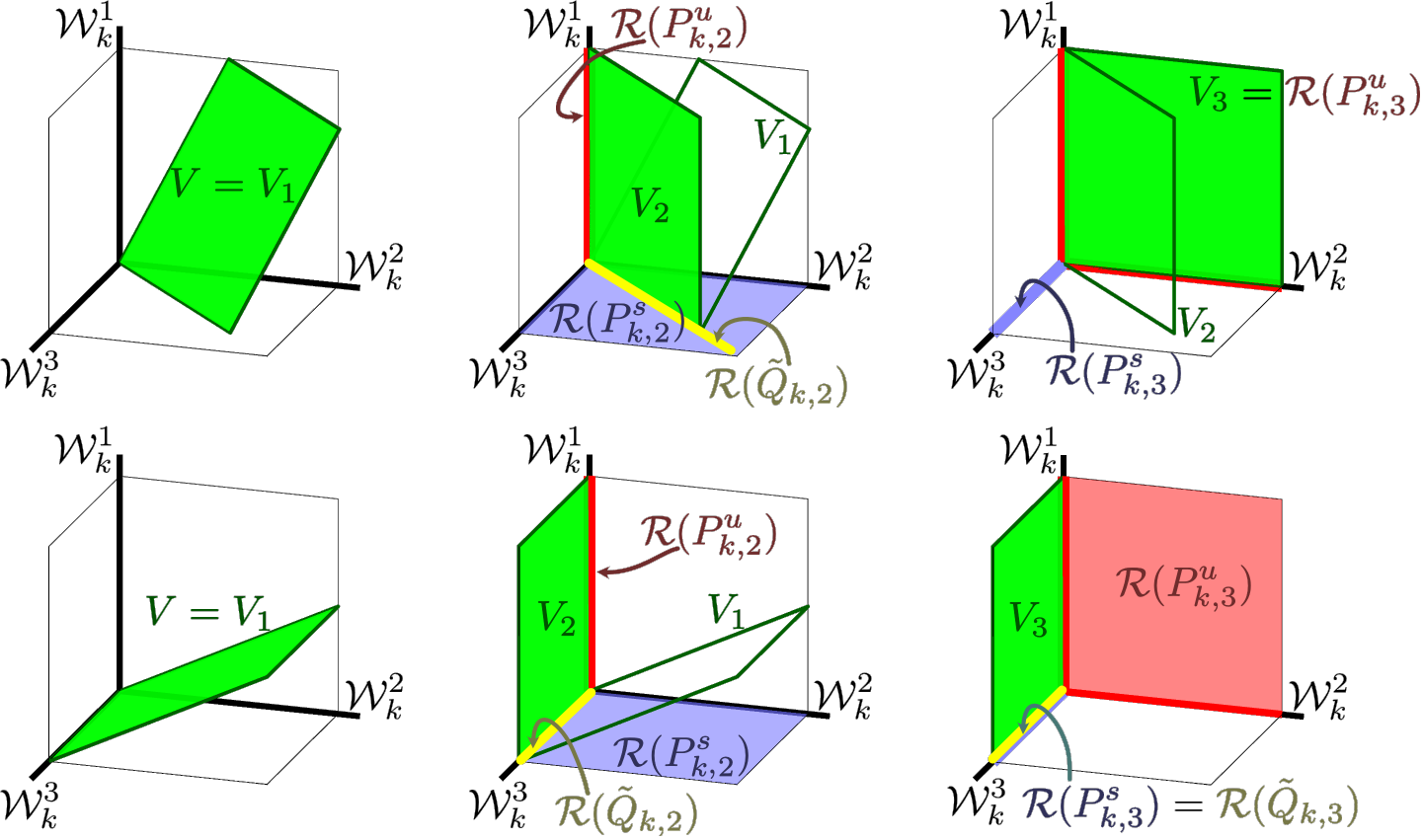}   
    \end{center}
\caption{Recursive construction of subspaces, cf.\ \eqref{defnewproj}
  for two characteristic cases. Upper row: Since $V$ intersects
  $\cW_k^i$ only in $\{0\}$ for any $i\in\{1,2,3\}$, the sequence of
  subspaces $(V_i)_{i\le 4}$ is 
  constant for $i\ge 3$, i.e.\ $V_2 \neq V_3 = V_4$. Lower row:
  $V$ has a nontrivial intersection with $\cW_k^3$ and the sequence of
  subspaces $(V_i)_{i\le 4}$ is constant for 
  $i \ge 2$. In both cases, the trace space of $V$ is given as
  $\cT_k(V) = V_3$.\label{Th24} }
\end{figure}

  Note that Lemma \ref{linalg} implies
  $V_{i+1}=(P_{k,i+1}^u+\tilde{Q}_{k,i+1}) V_i$ and $\dim V_i = \dim V_{i+1}$.
  In addition,  we claim for $i=1,\ldots,\ell$
  \begin{equation} \label{relproj}
    P_{k,i+1}^u V_i = \big(\sum_{\nu=1}^i \cP_{k,\nu} Q_{k,\nu}^s \big) V, \quad
    \tilde{Q}_{k,i+1}V_i = Q_{k,i+1}^s V.
  \end{equation}
  We proceed by induction.
  First note that \eqref{projhierarchy}, \eqref{defnewproj} and \eqref{Qkidef} imply
  $P_{k,1}^s = I_d$,  $Q_{k,1}V = \tilde{Q}_{k,1}V_1$. Further, we have
  by \eqref{fiberproj}, \eqref{defnewproj}, \eqref{Qkidef}
  \begin{align*}
    P_{k,2}^uV_1 = \cP_{k,1} V, \quad \tilde{Q}_{k,2}V_1 = \range(P_{k,2}^s) \cap V
    = Q_{k,2}^s V.
    \end{align*}
  Assume that \eqref{relproj} holds for the index $i$. Then we obtain from
  \eqref{defnewproj}, \eqref{dichbundle}
  \begin{align*}
    \tilde{Q}_{k,i+2}V_{i+1} &=\range(P_{k,i+2}^s) \cap V_{i+1} =
    \range(P_{k,i+2}^s) \cap (P_{k,i+1}^u V_{i} \oplus \tilde{Q}_{k,i+1} V_i) \\
    & = \range(P_{k,i+2}^s) \cap \tilde{Q}_{k,i+1} V_i
     = \range(P_{k,i+2}^s) \cap (\range(P_{k,i+1}^s) \cap V)\\ 
       &  =\range(P_{k,i+2}^s)  \cap V = Q_{k,i+2}^s V.
  \end{align*}
  Furthermore,
  \begin{align*}
    P_{k,i+2}^u V_{i+1} & = P_{k,i+2}^u( P_{k,i+1}^uV_i \oplus \tilde{Q}_{k,i+1}V_i)\\
    & = P_{k,i+2}^u \Big( \big(\sum_{\nu=1}^i \cP_{k,\nu} Q_{k,\nu}^s\big) V \oplus
    Q_{k,i+1}^sV\Big) \\
    & = \big(\sum_{\nu=1}^i P_{k,i+2}^u\cP_{k,\nu} Q_{k,\nu}^s\big) V +
    P_{k,i+2}^u P_{k,i+1}^s Q_{k,i+1}^s V.
  \end{align*}
  From \eqref{projhierarchy}, \eqref{dichbundle} and \eqref{fiberproj} we have the equalities
  $P_{k,i+2}^u\cP_{k,\nu}=\cP_{k,\nu}$ for $\nu \le i$ and\linebreak
  $P_{k,i+2}^u P_{k,i+1}^s= \cP_{k,i+1}$. With
  $\range(\cP_{k,i+1}) \cap \bigoplus_{\nu=1}^i\cW_k^{\nu}=\{0\}$ this leads to
  \begin{align*}
    P_{k,i+2}^u V_{i+1} & =
    \big(\sum_{\nu=1}^i \cP_{k,\nu} Q_{k,\nu}^s\big) V \oplus
    \cP_{k,i+1} Q_{k,i+1}^s V
    = \big(\sum_{\nu=1}^{i+1} \cP_{k,\nu} Q_{k,\nu}^s\big) V.
  \end{align*}
  The last equality needs an argument. The relation ``$\supseteq$'' is obvious.
  For the converse we consider $v,w\in V$ and construct $\tilde{v}\in V$
  such that
  \begin{equation} \label{eqsum}
    \sum_{\nu=1}^i \cP_{k,\nu} Q_{k,\nu}^s v + \cP_{k,i+1}Q_{k,i+1}^s w=
    \sum_{\nu=1}^{i+1} \cP_{k,\nu} Q_{k,\nu}^s \tilde{v}.
  \end{equation}
  For this purpose set $\tilde{v}=v + Q_{k,i+1}^s(w-v)$ and verify \eqref{eqsum}
  by using the equality
  $\cP_{k,\nu}Q_{k,\nu}^s Q_{k,i+1}^s(w-v)=\cP_{k,\nu} Q_{k,i+1}^s(w-v)=
  \cP_{k,\nu} P_{k,i+1}^s Q_{k,i+1}^s(w-v)=0$ for $\nu \le i$.
 In this way one also obtains the equality of the representations
 \eqref{tracedef} and 
  \eqref{tracerep} via an induction w.r.t.\ the index $i$.

  \noindent {\bf Step2:} We prove the key estimate \eqref{angletotrace}. 
 Let us apply Theorem \ref{Th1.1} for $i=1,\ldots,\ell$ to the scaled operator
  $\Phi_{\gamma}$ with $\gamma \in  R_{i+1}=(\sigma_{i+1}^+,\sigma_i^-)$
   and $V_i \in \cG(s,d)$, $\tilde{Q}_{k,i}$ as defined by
   \eqref{defnewproj} (recall \eqref{const_alpha} and
   $\sigma_{\ell+1}^+=0$). 
  The index $j^s_{k,i}$ is determined by  $2 K^2
  \big(\frac{\sigma_{i+1}^+}{\sigma_i^-}\big)^{j^s_{k,i}} \rho_{k,i}^s(V) \le
  1$ (cf.\ \eqref{indexlarge}) where, 
  due to the second equation in \eqref{relproj},
  \begin{align*}
    \rho^s_{k,i}(V)=  \inf\{C>0:
    \|P^s_{k,i+1}(I-Q_{k,i+1}^s)v  \| \le C
    \|P^u_{k,i+1}(I-Q_{k,i+1}^s) v \| \; \forall v \in V \}. 
  \end{align*}
  Then the estimate \eqref{estanguns} leads for $j \ge k +j_{k,i}^s$ and
  $i=1\ldots,\ell$ to
  \begin{equation*} \label{estVi}
     \ang( \Phi_{\gamma}(j,k)V_i,\Phi_{\gamma}(j,k)V_{i+1}))
    \le \frac{2}{\sqrt{3}}K^2
   \left(\frac{\sigma_{i+1}^+}{\sigma_i^-}\right)^{j-k} \rho^s_{k,i}(V).
  \end{equation*}
  Since angles do not depend on scalings we can replace $\Phi_{\gamma}$
  by $\Phi$ in this estimate.
  Finally, observe $P_{k,\ell+1}^s=0$, $Q_{k,\ell+1}^s=0$ and thus
  $V_{\ell+1}= P_{k,\ell+1}^u V_{\ell}=V_{\ell}=\cT_k(V)$ due to \eqref{relproj}.
  The triangle inequality then yields 
  for $j-k \ge \overline{j}= \max_{i=1,\ldots,\ell} j_{k,i}^s$
  \begin{align*}
    \ang(\Phi(j,k)V,\Phi(j,k)\cT_k(V)) & \le \sum_{i=1}^{\ell}
    \ang(\Phi(j,k)V_i,\Phi(j,k)V_{i+1}) \\
    & \le \frac{2K^2}{\sqrt{3}}\left(\max_{i=1,\ldots,\ell-1}\frac{\sigma_{i+1}^+}
        {\sigma_i^-}\right)^{j-k} \sum_{i=1}^{\ell} \rho_{k,i}^s(V).
    \end{align*}
\end{proof}

Some conclusions of Theorem \ref{Th2} are summarized in Theorem \ref{cor1}
below. 
In particular, we present an important characterization of outer angular
values $\underaccent{\hat}\theta_1$, $\hat \theta_1$ if
all spectral bundles are one-dimensional.

\begin{theorem}\label{cor1}
  Let the assumptions of Theorem \ref{Th2} hold and define the quantities
  (see \eqref{defsums})
  \begin{equation*} \label{defsums1}
    a_{1,n}(V) = \sum_{j=1}^{n} \ang(\Phi(j-1,0)V,\Phi(j,0)V) \quad 
   n \in \N,\; V\in \mathcal{G}(s,d).
  \end{equation*}
  Then the following holds for all $V\in \cG(s,d)$
\begin{equation}\label{obs1}
\begin{aligned}
  &\limsup_{n\to \infty}
  \frac 1n a_{1,n}(V)
= \limsup_{n\to \infty}
\frac 1n a_{1,n}(\cT_0(V)),
\end{aligned}
\end{equation}
and similarly with $\liminf$ instead of $\limsup$.
The outer angular values satisfy
\begin{equation} \label{outerreduction}
  \begin{aligned}
  \hat \theta_{s} = \sup_{V \in \cD_0(s,d)}
  \limsup_{n\to\infty} \frac{1}{n}  a_{1,n}(V), \quad
    \underaccent{\hat}\theta_s =\sup_{V \in \cD_0(s,d)}
  \liminf_{n\to\infty} \frac{1}{n}  a_{1,n}(V).
\end{aligned}
    \end{equation}
If $\dim(\cW_m^i)=1$ for all $i=1,\dots,d$ and $m \in \N_0$, then the
first lower and upper outer angular values have the form
\begin{equation}\label{obs2}
\begin{aligned}
\underaccent{\hat} \theta_1 & = \max_{i=1,\dots,d}\liminf_{n\to \infty} \frac 1n
\sum_{j=1}^{n} \ang(\cW_{j-1}^i,\cW_j^i),\\
\hat \theta_1 & = \max_{i=1,\dots,d}\limsup_{n\to \infty} \frac 1n
\sum_{j=1}^{n} \ang(\cW_{j-1}^i,\cW_j^i).
\end{aligned}
\end{equation}
\end{theorem}

\begin{proof}
  From the triangle inequality (Proposition \ref{prop2g})  we obtain
  \begin{equation} \label{partialest}
    \begin{aligned}
    |a_{1,n}(V) - a_{1,n}(\cT_0(V))| \le \sum_{j=1}^n&
    \left\{\ang(\Phi(j-1,0)V,\Phi(j-1,0)\cT_0(V))\right. \\
    + \left. \ang(\Phi(j,0)V,\Phi(j,0)\cT_0(V))\right\} &
    \le 2 \sum_{j=0}^n \ang(\Phi(j,0)V,\Phi(j,0)\cT_0(V)).
    \end{aligned}
  \end{equation}
    Theorem \ref{Th2} shows that the angles decay geometrically for
  $j \ge \overline{j}(0,V)$, hence the right-hand side is uniformly bounded by a
  constant depending on $V$ only. Therefore \eqref{obs1} follows, and
  \eqref{outerreduction} is an immediate consequence by taking the
  supremum with respect to 
  $V$.

  In the case $s=1$ and $\dim(\cW_m^i)=1$ for $i=1,\ldots,d$ the set
  $\cD_0(1,d)=\{ \cW_0^i: i =1,\ldots,d\}$ becomes finite. Moreover, we
  have $\Phi(j,0)\cW_0^i=\cW_j^i$ by the invariance condition \eqref{invar}.
  Thus, the formula \eqref{outerreduction} simplifies to \eqref{obs2}.
\end{proof}

\medskip

In view of Theorem \ref{Th2}, we revisit crucial examples from 
\cite[Section 3.2]{BeFrHu20}. 
The first model is defined for $n\in\N_0$ and
$0 \le \varphi_0 <\varphi_1 \le \frac \pi 2$ by
$$
A_n = 
\begin{cases}
\left(\begin{smallmatrix}
\cos(\varphi_0) & -\sin(\varphi_0)\\
\sin(\varphi_0) & \cos(\varphi_0)
\end{smallmatrix}\right),&
\text{ for } n = 0 \lor n \in \bigcup_{\ell=1}^\infty
[2^{2\ell-1},2^{2\ell} -1]\cap \N,\\[3mm]
\left(\begin{smallmatrix}
\cos(\varphi_1) & -\sin(\varphi_1)\\\sin(\varphi_1) & \cos(\varphi_1)
\end{smallmatrix}\right),&\text{ otherwise.}
\end{cases}
$$
For this example,  upper and lower angular values do not coincide in
general, more precisely, the diagram \eqref{eq2.4} now reads      
\begin{equation*}
    \begin{matrix}
    \underaccent{\hat} \theta_{[1]} & < & \underaccent{\hat}
    \theta_1 & < & \hat \theta_1 & < & \hat \theta_{[1]}\\
    \rg& & \rg && \rg && \rg\\
    \underaccent{\bar} \theta_{[1]} & < & \underaccent{\bar}
    \theta_1 & < & \bar \theta_1 & < & \bar \theta_{[1]}.
    \end{matrix}
\end{equation*}
For the second example, defined for $n\in\N_0$ by 
$$
A_n :=
\begin{cases}
\left(\begin{smallmatrix}
-1 & 0 \\ 0 & 1
\end{smallmatrix}\right), 
& \text{ for } n \in \bigcup_{\ell = 1}^\infty [2\cdot 2^\ell -
4,3\cdot 2^\ell -5],\\
\left(\begin{smallmatrix}
1 & 0 \\ 0 & \frac 12
\end{smallmatrix}\right), 
&\text{ otherwise}
\end{cases}
$$
inner and outer angular values differ, i.e.\ the diagram \eqref{eq2.4} turns into
\begin{equation*}
    \begin{matrix}
    \underaccent{\hat} \theta_{[1]} & = & \underaccent{\hat}
    \theta_1 & = & \hat \theta_1 & = & \hat \theta_{[1]}\\
    \rg& & \rl && \rl && \rl\\
    \underaccent{\bar} \theta_{[1]} & < & \underaccent{\bar}
    \theta_1 & < & \bar \theta_1 & < & \bar \theta_{[1]}.
    \end{matrix}
\end{equation*}
The dichotomy spectrum of the first example is given by $\Sigma_\ED =
\{1\}$ and for the second example, we obtain $\Sigma_\ED = [\frac 12,
1]$. In both cases $\cW_k^1 = \R^2$ for all $k\in\N$. Thus
one-dimensional trace spaces agree with the given space. In
particular, the detection of angular values cannot be reduced by
Theorem \ref{Th2} and Theorem \ref{cor1} to lower dimensional spaces. 
       
\subsection{Inner angular values and spectral bundles}
\label{sec4.3a}
Inner angular values are more difficult to handle,  both numerically
and theoretically, since the supremum over all subspaces is taken before
going to the limit.
For general dimensions we do not have a result comparable to
Theorem \ref{Th2}. However, for one-dimensional subspaces a reduction
is possible under a uniformity condition. Recall from \eqref{defsums} the notion
\begin{align} \label{abbsum}
   a_{m,n}(v) = \sum_{j=m}^{n} \ang(\Phi(j-1,0)v,\Phi(j,0)v) \quad 
   m,n \in \N,\; v \in \R^d, v\neq 0
\end{align}
with $a_{m,n}(v)=0$ for $m>n$. 
For a subspace $V \subseteq \R^d$ we introduce the quantity
\begin{align*}
  \overline{\theta}_1(V)= \limsup_{n \to \infty} \sup_{v \in V, v \neq 0}\frac{a_{1,n}(v)}{n}.
\end{align*}
  Note that  $\overline{\theta}_1(V)$ is the maximum angular value
  for all one-dimensional subspaces of $V$. In case $V=\R^d$
  this value agree with $\overline{\theta}_1$ as
  defined in \eqref{dinner}.
\begin{theorem} \label{Th4.1}
  Let the assumptions of Theorem \ref{Th2} hold. Further assume that
  the inner and the uniform inner angular values (cf.\ \eqref{dinner},
  \eqref{duniinner} and \eqref{eq2.4}) agree within each fiber,
  i.e.\ for $i=1,\ldots,\ell$ the following holds
  \begin{align} \label{unifiber}
    \overline{\theta}_1(\cW_0^i)=
  \lim_{n\to\infty} \frac{1}{n} \sup_{v \in \cW_0^i, v \neq 0}
  \sup_{k \in \N_0} a_{k+1,k+n}(v).
  \end{align}
  Then the first inner angular value
     $\overline{\theta}=\overline{\theta}_{1}(\R^d)$ satisfies
  \begin{equation*} \label{theta1inner}
    \overline{\theta}_1= \max_{i=1,\ldots,\ell} \overline{\theta}_1(\cW_0^i).
    \end{equation*}
\end{theorem}
\begin{proof}
  The main step is to show for $i=1,\ldots,\ell$
  \begin{equation} \label{thetai}
    \overline{\theta}_1(\range(P_{0,i}^s)) 
    \le \max\left(\overline{\theta}_1(\cW_0^i),
   \overline{\theta}_1(\range(P_{0,i+1}^s))\right).
  \end{equation}
  Since $P_{0,1}^s=I_d$, $P_{0,\ell+1}^s=0$ and $\sup_{\emptyset}=0$,
  we obtain by induction 
  \begin{align*}
    \overline{\theta}_1(\R^d)
    \le \max_{i=1,\ldots,\ell} \overline{\theta}_1(\cW_0^i).
  \end{align*}
  The  converse inequality ``$\ge$'' is obvious, hence our assertion
  is proved.

  In the following we choose $j_{\star}$ such that (cf.\ \eqref{angletotrace})
  \begin{align*}
    2 K^2 q^{j_{\star}} \le 1, \quad \text{where} \quad
    q:= \max_{i=1,\ldots,\ell}
  \frac{\sigma_{i+1}^+}{\sigma_i^-} <1.
\end{align*}
  
  For the proof of \eqref{thetai} it is enough to consider $v \in \range(P_{0,i}^s)$
  with $v \notin \range(P_{0,i+1}^s)$ and $v \notin \cW_0^i$.
  Figure \ref{idea_proof} illustrates the idea of the proof.
  Vectors $v$ close to but not in $\range(P_{0,i+1}^s)$
    may spend arbitrarily large time near $\range(P_{0,i+1}^s)$ until finally converging to $\range(P_{0,i+1}^u)$. There is an exponentially growing initial phase switching at
    some index $k_{\star}$ to an exponentially decreasing final phase.
    Though the switching point $k_{\star}$ depends on $v$ our uniformity
    assumption will allow estimates independent of $v$.
  \begin{figure}[hbt]
 \begin{center}
   \includegraphics[width=0.60\textwidth]{ 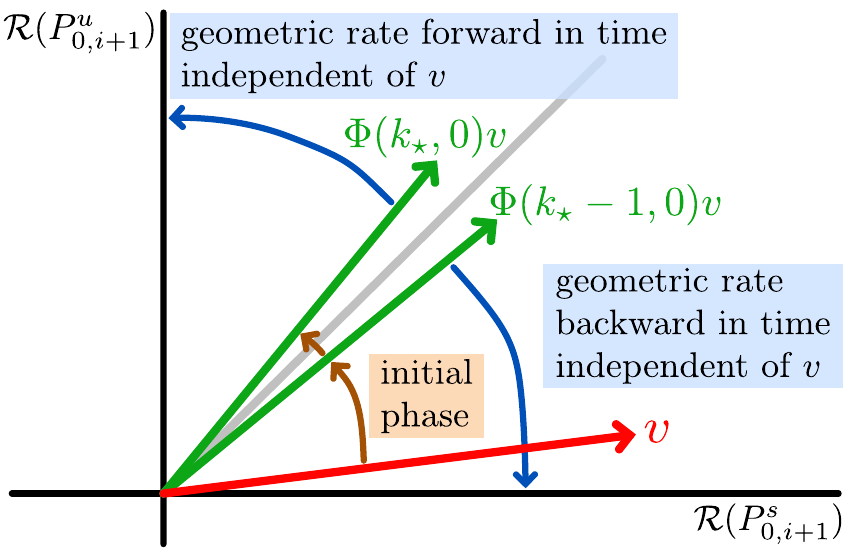}
 \end{center}
 \caption{\label{idea_proof}Idea of proof for Theorem
   \ref{Th4.1}. First, we construct a $v$-dependent index $k_\star$,
   such that $\Phi(k_\star,0)v$ is just above the diagonal. Then, we
   exploit geometric convergence forward resp.\ backward in time for
   proving $v$-independent estimates.}
\end{figure}
   We choose $\gamma \in R_{i+1}$
  and apply Theorem \ref{Th1.1} to $\Phi_{\gamma}$ and $V=\mathrm{span}(v)$
  as in Step 2 of the proof of Theorem \ref{Th2}. Note that the
  projector $Q_{0,i+1}^s:\R^d \to \range(P_{0,i+1}^s)\cap V$ from
  \eqref{Qkidef} 
  is trivial since $V$ is one-dimensional and $v \notin \range(P_{0,i+1}^s)$.
  Moreover, by  \eqref{fiberproj} we have $\cP_{0,i}v= P_{0,i+1}^u
  P_{0,i}^sv= P_{0,i+1}^uv \neq 0$. 
  Therefore, we can invoke inequality \eqref{powerest}
  from the proof of Theorem \ref{Th1.1} with $\tilde{P}=P_{k,i+1}^u$,
  $\rho_k^s(v)=\frac{\|P_{0,i+1}^sv\|}{\|P_{0,i+1}^uv\|}$, $\alpha_s
  \alpha_u \le q$. This shows  
  \begin{align*}
    \|P_{k,i+1}^s \Phi_{\gamma}(k,0)v\| \le K^2 q^k\rho_k^s(v) \|
    P_{k,i+1}^u \Phi_{\gamma}(k,0)v\| \quad \forall k \in \N_0. 
  \end{align*}
  We conclude that the following index, depending on $v$, exists
  \begin{align*}
    k_{\star}=k_{\star}(v) = \min\{ k \in \N_0:\|P_{k,i+1}^s \Phi_{\gamma}(k,0)v\|
    \le \| P_{k,i+1}^u \Phi_{\gamma}(k,0)v\|\}.
  \end{align*}
  Applying Theorem \ref{Th1.1} once more to
  $\Phi_{\gamma}$ and $V= \mathrm{span}(\Phi_{\gamma}(k_{\star},0)v)$ then shows
  for $j \ge k_{\star}+j_{\star}$
  \begin{equation*} \label{forwardest}
    \begin{aligned}
      & \ang(\Phi(j,k_{\star}) \Phi(k_{\star},0)v,
      \Phi(j,k_{\star})P_{k_{\star},i+1}^u \Phi(k_{\star},0)v) \\
      & = \ang(\Phi(j,0)v,\Phi(j,0)P_{0,i+1}^uv)
      \le \frac{2 K^2}{\sqrt{3}} q^{j-k_{\star}}.
    \end{aligned}
  \end{equation*}
  Next we estimate angles for $j \le k_{\star}$ by invoking \eqref{estangstab}
  with $k = k_{\star}-1$ and $V=\mathrm{span}(\Phi_{\gamma}(k_{\star}-1,0)v)$.
  Since $v \notin \cW_0^i$ the projector in \eqref{Qkudef} is trivial,
  and \eqref{indexularge} holds by the choice of $j_{\star}$. 
  Hence we obtain for $0\le j \le k_{\star}-1-j_{\star}$
  \begin{equation*} \label{backwardest}
    \begin{aligned}
     & \ang(\Phi(j,k_{\star}-1) \Phi(k_{\star}-1,0)v,
      \Phi(j,k_{\star})P_{k_{\star},i+1}^s \Phi(k_{\star},0)v)\\
    &= \ang(\Phi(j,0)v, \Phi(j,0)P_{0,i+1}^s v) \le \frac{2 K^2}{\sqrt{3}}
    q^{k_{\star}-1-j}.
    \end{aligned}
  \end{equation*}
  Combining these estimates with the triangle inequality we find with
  suitable constants $C$ independent of $n,j,v$,
  \begin{equation} \label{criticalest}
    \begin{aligned}
   \frac{a_{1,n}(v)}{n} &= \frac{1}{n}\sum_{j=1}^{n} \ang(\Phi(j-1,0)v,\Phi(j,0)v)\\
   & \le \frac{1}{n} \Big[\big\{ \sum_{j=1}^{k_{\star}-1-j_{\star}} +
   \sum_{j=k_{\star}+1+j_{\star}}^n 
     \big\}\ang(\Phi(j-1,0)v,\Phi(j,0)v) + (2j_{\star}+1) \frac{\pi}{2} \Big]\\
   &\le \frac{1}{n} \Big[ \frac{4K^2}{\sqrt{3}}( \sum_{j=0}^{k_{\star}-1-j_{\star}}
       q^{k_{\star}-1-j}+ \sum_{j=k_{\star}}^n q^{j-k_{\star}})+C \\ 
      & + \sum_{j=1}^{k_{\star}}\ang(\Phi(j-1,0)P_{0,i+1}^sv,\Phi(j,0)P_{0,i+1}^sv)
       \\
       & + \sum_{j=k_{\star}+1}^n\ang(\Phi(j-1,0)P_{0,i+1}^uv,\Phi(j,0)P_{0,i+1}^uv)
       \Big]\\
   & \le \frac{1}{n}\left[ C + a_{1,k_{\star}}(P_{0,i+1}^sv) +
     a_{k_{\star}+1,n}(P_{0,i+1}^u v) \right].
    \end{aligned}
  \end{equation}
  Given $\varepsilon>0$, assumption \eqref{unifiber} yields a number
  $n_0=n_0(\varepsilon)$ such that 
  \begin{equation} \label{epscond}
 \begin{aligned}
   \frac{1}{n} \sup_{k \in \N_0} \sup_{v \in \cW_0^i}a_{k+1,k+n}(v) &\le
   \overline{\theta}_1(\cW_0^i) + \varepsilon, \quad \forall n \ge n_0,\\
    \frac{1}{n} \sup_{v \in \range(P_{0,i+1}^s)} a_{1,n}(v) & \le
   \overline{\theta}_1(\range(P_{0,i+1}^s))+\varepsilon, \quad \forall n \ge n_0.
 \end{aligned}
 \end{equation}
  Thus we have
  \begin{equation*} \label{estvarious}
 \begin{aligned}
   a_{k_{\star}+1,n}(P_{0,i+1}^uv)& \le
   \begin{cases} (n-k_{\star})(\overline{\theta}_1(\cW_0^i) + \varepsilon),
     & \text{if } n - k_{\star} \ge n_0, \\
     n_0\frac{\pi}{2}, & \text{if } n - k_{\star} < n_0,
     \end{cases}
     \\
   a_{1,k_{\star}}(P_{0,i+1}^sv) & \le
   \begin{cases} k_{\star} (\overline{\theta}_1(\range(P_{0,i+1}^s))+\varepsilon),
     & \text{if } k_{\star} \ge n_0, \\
     n_0 \frac{\pi}{2}, & \text{if } k_{\star} < n_0.
     \end{cases}
 \end{aligned}
 \end{equation*}
 Summing up, we obtain for $n \ge n_0(\varepsilon)$,
 \begin{align*}
   \frac{a_{1,n}(v)}{n} & \le \frac{1}{n} \left[ 
      \min(k_{\star},n)(\overline{\theta}_1(\range(P_{0,i+1}^s))+\varepsilon)
      + (n-\min(n,k_{\star}))(\overline{\theta}_1(\cW_0^i) + \varepsilon)
      \right. \\
   & \left.  \qquad + C + n_0\pi \right] \le \max(\overline{\theta}_1(\cW_0^i),
   \overline{\theta}_1(\range(P_{0,i+1}^s))) +\varepsilon +
   \frac{1}{n}(C +  n_0\pi).
 \end{align*}
 Finally, the assertion \eqref{thetai} follows by taking the supremum over $v$
 and making the last term small for $n$ sufficiently large.
\end{proof}
       
\subsection{Uniform angular values and spectral bundles}
\label{sec4.4}
In this section we extend Theorem \ref{cor1} and Theorem \ref{Th4.1}
to uniform 
outer and inner angular values. As before, we show that it is enough
to compute angular values for subspaces which have their basis in
the fibers induced by the dichotomy spectrum. Since we deal with
uniform angular values a uniformity condition like \eqref{unifiber} is
no longer needed.

\begin{theorem}\label{Th4.2}
  Let the assumptions of Theorem \ref{Th2} hold. Then the uniform
  outer angular values $\hat
  \theta_{[s]},\underaccent{\hat}\theta_{[s]}$, $s=1,\ldots,d$,  can
  be represented with the partial sums \eqref{defsums} 
  and the  set of trace spaces \eqref{tracespace}  as follows:
 \begin{align}
\label{sup-eq}
\hat \theta_{[s]} &=
\sup_{V\in\cD_0(s,d)}
\lim_{n\to \infty} \sup_{k\in\N_0} \frac 1n
a_{k+1,k+n}(V),\\
\label{inf-eq}
  \underaccent{\hat}\theta_{[s]} &=
 \sup_{V\in\cD_0(s,d)}                                  
\lim_{n\to \infty} \inf_{k\in\N_0} \frac 1n
a_{k+1,k+n}(V).
 \end{align}
 With the partial sums from \eqref{abbsum}, the first uniform inner
 angular value satisfies 
 \begin{equation} \label{unisupeq}
   \overline{\theta}_{[1]}= \max_{i=1,\ldots,\ell} \overline{\theta}_{[1]}(\cW_0^i),
   \quad \text{where}\quad \overline{\theta}_{[1]}(V)=
   \lim_{n \to \infty} \sup_{v \in V}\sup_{k \in \N_0} \frac{1}{n} a_{k+1,k+n}(v).
   \end{equation}
\end{theorem}
\begin{proof} 
  Recall $a_{m,n}(V)$ from \eqref{defsums} and use
  \eqref{partialest}, \eqref{angletotrace} to find that $V\in \cG(s,d)$
  satisfies with some constant $C$ depending on $V$ but not on $k,n$,
  \begin{equation} \label{unitrace}
    \begin{aligned}
    |a_{k+1,k+n}(V) - a_{k+1,k+n}(\cT_0(V))| & \le 2 \sum_{j=k}^{k+n}
    \ang(\Phi(j,0)V,\Phi(j,0)(\cT_0(V))) \\
    & \le 2 C(0,V)  \sum_{j=k}^{k+n} q^{j-k} \le C.
    \end{aligned}
  \end{equation}
    Given $\varepsilon>0$, choose $n_0$ such that for $n \ge n_0$
  \begin{align*}
    \big| \frac{1}{n} \sup_{k\in \N_0}a_{k+1,k+n}(V) - \lim_{m\to \infty}
    \frac{1}{m} \sup_{k\in \N_0}a_{k+1,k+m}(V)\big| \le \varepsilon.
  \end{align*}
  Then select $k(n)\in \N_0$ such that
  $\frac{1}{n} |a_{k(n)+1, k(n)+n}(V) - \sup_{k \in \N_0} a_{k+1,k+n}(V)|
  \le \varepsilon$ holds for $n\ge n_0$. This implies
  \begin{align*}
    \big| \frac{1}{n} a_{k(n)+1, k(n)+n}(V)-\lim_{m\to \infty}\frac{1}{m}
    \sup_{k\in \N_0}a_{k+1,k+m}(V)\big| \le 2 \varepsilon, \quad n \ge n_0.
  \end{align*}
  With \eqref{unitrace} we obtain  for $n \ge n_0$ 
  \begin{align*}
    \lim_{m\to \infty}  \frac{1}{m} \sup_{k\in \N_0}a_{k+1,k+m}(V) & \le
    \frac{1}{n} a_{k(n)+1, k(n)+n}(V) + 2 \varepsilon \\
    & \le \frac{C}{n} +\frac{1}{n} a_{k(n)+1, k(n)+n}(\cT_0(V)) + 2 \varepsilon\\
    & \le \frac{C}{n} +\frac{1}{n} \sup_{k \in \N_0} a_{k+1,
      k+n}(\cT_0(V)) + 2 \varepsilon. 
  \end{align*}
  As $n \to \infty$ this shows
  \begin{align*}
    \lim_{n\to \infty} \frac{1}{n} \sup_{k\in \N_0}a_{k+1,k+n}(V) \le 2 \varepsilon
    + \lim_{n\to \infty} \frac{1}{n} \sup_{k\in \N_0}a_{k+1,k+n}(\cT_0(V)).
  \end{align*}
  A corresponding inequality with $V$ and $\cT_0(V)$ exchanged, is
  proved in the same manner, and \eqref{sup-eq} 
  follows by taking the supremum over $V \in \cG(s,d)$.
  The same type of estimate leads to \eqref{inf-eq}.
  The formula in \eqref{unisupeq} follows by adapting the proof of
  Theorem \ref{Th4.1}. 
  For the quantities
    $\overline{\theta}_{[1]}(V)$  from \eqref{unisupeq}
  we show
  \begin{equation} \label{unithetai}
    \overline{\theta}_{[1]}(\range(P_{0,i}^s)) 
    \le \max\left(\overline{\theta}_{[1]}(\cW_0^i),
   \overline{\theta}_{[1]}(\range(P_{0,i+1}^s))\right), \quad i=1,\ldots,\ell.
  \end{equation}
  The estimate \eqref{criticalest} for
  $v \in \range(P_{0,i}^s)\setminus(\range(P_{0,i+1}^s)\cup\cW_0^i)$ now reads
  \begin{align*}
    \frac{1}{n}a_{k+1,k+n}(v) & \le \frac{1}{n} \left[ C +
                                a_{k+1,k_{\star}-j_{\star}-1}(P_{0,i+1}^sv)
                                +
                                a_{\max(k,k_{\star}+j_{\star}+1),n+k}
                                (P_{0,i+1}^uv) \right]. 
  \end{align*}
  Recall $a_{m,n}=0$ for $m>n$ and note that there is no relation between
  $k, k_{\star}(v)$ and $n$. The condition \eqref{epscond} for $n_0$ turns into
  \begin{equation*} \label{uniepscond}
  \begin{aligned}
   \frac{1}{n}  \sup_{v \in \cW_0^i} \sup_{k \in \N_0}a_{k+1,k+n}(v) &\le
   \overline{\theta}_{[1]}(\cW_0^i) + \varepsilon, \quad \forall n \ge n_0,\\
    \frac{1}{n}\sup_{v \in \range(P_{0,i+1}^s)} \sup_{k \in \N_0} a_{k+1,k+n}(v) & \le
   \overline{\theta}_{[1]}(\range(P_{0,i+1}^s))+\varepsilon, \quad
   \forall n \ge n_0. 
 \end{aligned}
  \end{equation*}
  This  leads to
  \begin{equation} \label{estunifvarious}
 \begin{aligned}
   a_{k_1,n+k}(P_{0,i+1}^uv)& \le
   \begin{cases} (n+k-k_1)(\overline{\theta}_{[1]}(\cW_0^i) + \varepsilon),
     & \text{if }n+ k - k_1 \ge n_0, \\
     n_0\frac{\pi}{2}, & \text{if } n+ k - k_1  < n_0,
     \end{cases}
     \\
   a_{k+1,k_2}(P_{0,i+1}^sv) & \le
   \begin{cases}( k_2-k-1)
     (\overline{\theta}_{[1]}(\range(P_{0,i+1}^s))+\varepsilon), 
     & \text{if } k_2-k-1 \ge n_0, \\
     n_0 \frac{\pi}{2}, & \text{if }k_2-k-1  < n_0,
     \end{cases}
 \end{aligned}
 \end{equation}
 where $k_1= \max(k,k_{\star}+j_{\star}+1)$, $k_2=k_{\star}-j_{\star}-1$.
 In both cases $k \ge k_{\star}+j_{\star}+1$ and $k < k_{\star}+j_{\star}+1$
 we find the estimate $n+k - k_1 + k_2 - k -1 \le n$ for the sum of
 coefficients in \eqref{estunifvarious}. Hence, we can continue 
 \begin{align*}
   \frac{1}{n}a_{k+1,k+n}(v)&  \le \frac{1}{n} \left[ C + n_0 \pi
     + n \left(\max(\overline{\theta}_{[1]}(\cW_0^i),
    \overline{\theta}_{[1]}(\range(P_{0,i+1}^s))) + \varepsilon\right) \right]. 
 \end{align*}
 Taking the supremum over $k$ and $v$ and then letting $n \to \infty$ yields
 the assertion \eqref{unithetai} as in the proof of Theorem \ref{Th4.1}.
 \end{proof}


\section{Numerical algorithms and results}
\label{sec5}
The aim of this section is to develop an algorithm for the
numerical detection of outer angular values.
The previous section provides the essential reduction result in
Theorem \ref{cor1}:
\begin{align*}
\hat \theta_s &= \sup_{V \in \cG(s,d)} \limsup_{n\to\infty} \frac 1n
\sum_{j=1}^n \ang(\Phi(j-1,0)V, \Phi(j,0)V)\\
&= \sup_{V \in \cD_0(s,d)} \limsup_{n\to\infty} \frac 1n
   \sum_{j=1}^n \ang(\Phi(j-1,0)V, \Phi(j,0)V).
\end{align*}
The search for the supremum of $V$ in $\cD_0(s,d)$ instead of
$\cG(s,d)$ reduces the computational effort substantially and, in some
cases, one needs to consider only finitely many subspaces.
This reduction method receives further support from the fact that some obvious
numerical approaches tend to fail:
\begin{itemize}
\item Algorithms based on a simple forward
  iteration cannot provide the largest angular value. 
  A generic subspace is pushed by the
dynamics towards the most unstable trace space of equal
dimension, see the upper row in Figure \ref{Th24}. 
But in general, the angular value is not achieved in this subspace.
For the correct angular value, also non-generic subspaces
must be considered, as sketched in the lower row of 
Figure \ref{Th24}. We refer to the H\'{e}non example
  in Section \ref{sec_hen} and to Figure \ref{naiv} which
  illustrates the failure of a naive approach. 
\item Algorithms based on the computation of eigenvalues and
  eigen\-spaces, e.g.\ by applying the Schur decomposition,
  provide good results for
  autono\-mous systems. This fits well to our theory, since in the
  autonomous case, spectral bundles w.r.t.\ 
  the dichotomy spectrum are indeed eigenspaces.
The corresponding analysis for autonomous systems is carried out in
detail in \cite{BeFrHu20}. 
  For nonautonomous models, eigenvalues of linearizations at a fixed time are
  known to be dynamically
  irrelevant as was first shown by Vinograd; see
  \cite{v52}. Corresponding algorithms fail in testing all trace
  subspaces.
\end{itemize}

To resolve these issues, we first detect the dichotomy spectrum and
the corresponding spectral bundles. Then all trace subspaces become
available, resulting in a numerically expensive but reliable
approximation of $\hat \theta_s$.  Finally, let us emphasize
  that this value aims at finding the subspace of maximal rotation in
  the global attractor of the system (if it exists). In general, our
  approach to angular values ignores other dynamically relevant features of
  the system, such as invariant manifolds or unstable fixed points
  and unstable periodic orbits if they do not belong to the $\omega$-limit set
  of the current trajectory.
 
\subsection{An algorithm for computing angular
  values}\label{algorithm}
Consider the difference equation \eqref{diffeq} on an
interval $\I$ that is bounded from below. 
We propose the following three steps for the numerical approximation of
$\hat \theta_s$. 

First, we compute the dichotomy spectrum in step 1 followed by an
approximation of the
corresponding fiber bundles in step 2. To obtain accurate numerical results
on the discrete interval $[0,M]\cap \N_0$, buffer intervals of length
$b$ (we choose $b=50$) are needed.
The algorithm from \cite{fhmw13, h17} requires to solve least squares 
problems on the extended interval
$\I=[-b,M+b]\cap \N_0$. 

The crucial part of our algorithm is the approximation of $\hat
\theta_s$ for $s\in\{1,2\}$ in step 3. It is based on the reduction 
results from Section \ref{sec4} and therefore needs the 
spectral bundles that are computed in step 2. 
Readers, familiar with the computation of dichotomy spectra and
spectral bundles may proceed directly with step 3.  
 
\subsubsection*{Step 1: Computation of the dichotomy spectrum}

The computation of Bohl exponents leads to an efficient algorithm for
the approximation of the dichotomy spectrum.
Upper and lower Bohl exponents of the scalar difference equation
\begin{equation*}\label{scalar}
u_{n+1} = a_n u_n,\quad n\in\I,\quad 0 < \inf_{n\in\I} |a_n| \le
\sup_{n\in\I}|a_n| < \infty
\end{equation*}
are defined as, see \cite{hp14}
$$
\underline{\beta}(a_\I):=\lim_{n\to\infty}\inf_{\kappa\in\I}
\left(\prod_{j=\kappa}^{\kappa+n-1}|a_j|\right)^{\frac 1n},\qquad
\overline{\beta}(a_\I):=\lim_{n\to\infty}\sup_{\kappa\in\I}
\left(\prod_{j=\kappa}^{\kappa+n-1}|a_j|\right)^{\frac 1n}.
$$
It follows that $\Sigma_\ED =[ \underline{\beta}(a_\I),
\overline{\beta}(a_\I)]$.

For the $d$-dimensional difference equation \eqref{diffeq}, a
corresponding result is more 
delicate to obtain. One may first transform the system into upper triangular
form, using a \texttt{qr}-decomposition $A = QT$ of a given matrix $A$
into the 
product of an orthogonal matrix $Q$ and an upper triangular matrix $T$, see
\cite[Section 4.4]{hu09} and \cite{dev10}: 
\medskip
  
\begin{center}
 \begin{minipage}{.3\linewidth}
\begin{algorithmic}
\State{$Q_0 T_0 = \texttt{qr}(A_0)$}
\For {$j = 1,2,\dots$}
\State{$Q_j T_j = \texttt{qr}(A_j Q_{j-1})$}
\EndFor
\end{algorithmic}
\end{minipage}
\end{center}
\medskip

Note that $A_j = Q_j T_j Q_{j-1}^\top$ for $j\ge 1$.
For non-degenerate models, the Bohl exponents of the diagonal
entries of $T_j$ (denoted by $T_j(i,i)$)
determine the dichotomy spectrum.
We refer to \cite{p16} for details
on the corresponding theoretical background. More relations of
Bohl exponents to other exponents and a perturbation analysis may be
found  for discrete-time systems in \cite{Barabanov2017} and for continuous-time
systems in \cite{Barabanov2001}.
In our case we fix a sufficiently large $H\in\N$ (called the Steklov window)
and compute
\begin{equation}\label{betaBohl}
\beta(i,\kappa) := \left(\prod_{j=\kappa}^{\kappa+H} |T_j(i,i)|\right)^{\frac
  1H},\quad i = 1,\dots,d,\quad \kappa=0,1,\dots. 
\end{equation}
With $\underline{\beta}(i) := \min_{\kappa} \beta(i,\kappa)$,
$\overline{\beta}(i) := \max_{\kappa} \beta(i,\kappa)$ we obtain the approximate spectrum
$\Sigma_\ED \approx \bigcup_{i=1}^d [ \underline{\beta}(i),
  \overline{\beta}(i)]$ where the values are ordered according to
$\underline{\beta}(1)\ge \cdots \ge \underline{\beta}(d)$.
In the numerical experiments the spectral
values turned out to be rather insensitive to the choice of Steklov window,
and $H=\lfloor\frac M2\rfloor$ was found to be suitable 
for all results below.

\subsubsection*{Step 2: Computation of spectral bundles}
Recall for $j\in\N$ and $i\in\{1,\dots,\ell\}$ the
representation \eqref{specbun}, \eqref{fiberproj} of the spectral
bundle $\cW_j^i = \range(\cP_{j,i}) = \range (P_{j,i}^s P_{j,i+1}^u)$
with $\dim(\cW_j^i) = d_i$. For computing these sets numerically, we
apply the ansatz, proposed in 
\cite[Section 2.5]{h17}. Take $d_i$ random vectors $r_{\nu}\in\R^d$
and obtain a basis of $\cW_j^i$ (in a generic sense) by
calculating  $\cP_{j,i}r_{\nu}$ for $\nu=1,\dots,d_i$. For this task, we
choose $\gamma_i \in R_i$, $\gamma_{i+1}\in R_{i+1}$ 
close to the boundary of the spectral interval $\cI_i$, cf.\
\cite[Section 2.6]{h17}.
We solve for each $\nu\in \{1,\dots,d_i\}$ and simultaneously for
$j\in [0,M]\cap \N_0$ the inhomogeneous linear systems
\begin{equation}\label{bundle}
\begin{array}{rcl}
v_{n+1}^{\nu} &=& \displaystyle\frac 1{\gamma_{i+1}} A_n v_n^{\nu} +
            \delta_{n,j-1}r_{\nu},\\[4mm]
u_{n+1}^{\nu} &=&  \displaystyle\frac 1{\gamma_i} A_n u_n^{\nu}
  -\delta_{n,j-1}A_{j-1}v_{j-1}^{\nu},
\end{array}\quad n= -b,\dots,M+b-1           
\end{equation}
in a least squares sense. Here,
$\delta$ denotes the Kronecker symbol.
For the solutions of \eqref{bundle} one has $\cP_{j,i}r_{\nu} \approx
u_j^{\nu}$, 
and we refer to \cite[Section 2.6]{h17} for precise error estimates.
In this way, we obtain bases
  of $\cW_j^i$ for $j\in [0,M]\cap \N_0$ and $i\in\{1,\dots,\ell\}$.
    If these fiber bundles are two-dimensional, we choose  an orthonormal basis
  at each time instance. Note that an accurate computation of
  $\Phi(j,0)v$ in \eqref{winkelrechnung} for $v\in \cB_0^i$ can only be achieved by projecting
  the results to the respective spectral bundles. Thus, it does not
  suffice to have an approximation of the initial fiber
  $\cW_0^i$, only.  

\subsubsection*{Step 3: Computation of angular values}
Assume that the
spectral bundles $\cW_{j}^i$, $i=1,\dots,\ell$, $j=0,\dots,M$ have
been computed in step 2.
We present a numerical scheme for computing approximate values of
$\hat \theta_s$ in the case
$s\in\{1,2\}$. Assume $\dim(\cW_0^i) \in \{1,2\}$  and introduce
the balls 
$\cB_0^i=\{v\in \cW_0^i: \|v\|=1 \}$ for all $i\in\{1,\dots,\ell\}$.
For a subspace $V\in \cG(s,d)$ (resp. a
vector $v \in \R^d$ ) we abbreviate as in \eqref{defsums}
\begin{align}\label{winkelrechnung}
  \theta_s(V) = \frac{1}{M} a_{1,M}(V)= \frac 1M \sum_{j=1}^{M}
\ang(\Phi(j-1,0)V,\Phi(j,0)V), \quad \theta_s(v)= \theta_s(\mathrm{span}(v)).
\end{align}
Our goal is to use the $M$-dependent values $\theta_s(V)$ for 
an approximation of
\begin{equation}\label{thetafinal}
\hat \theta_s \approx \hat \theta_{s,M} := \sup_{V\in\cG(s,d)}  \theta_s(V).
\end{equation}
Starting with $s=1$ our scheme reads:
\medskip

\begin{center}
\begin{minipage}{.30\linewidth}
\begin{algorithmic}
\For {$i = 1,\dots,\ell$}
\State{$\displaystyle w^i = \max_{v \in \cB_0^i}\theta_1(v)$}
\EndFor
\State{$\hat \theta_1 = \max_{i=1,\dots,\ell} w^i$.}
\end{algorithmic}
\end{minipage}
\end{center}
\medskip

If $\dim(\cW_0^{i}) = 1$ then 
$\theta_1(\cW_0^{i})$ is computed for a single
one-dimensional subspace.
The detection of $\max_{v\in\cB_0^i}\theta_1(v)$ is a
one-dimensional optimization problem if $\dim \cW_0^i=2$.
For this task, we apply the
MATLAB-routine \texttt{fminbnd} that is based
on golden section search and parabolic interpolation.
The corresponding scheme for $s=2$ is given by:
\medskip

\begin{center}
\begin{minipage}{.5\linewidth}
\begin{algorithmic}
\State{$\kappa = 0$}
\For{$i = 1,\dots,\ell$}
\If{$\dim(\cW_0^i) = 2$}
\State{$\kappa = \kappa + 1$}
\State{$w^\kappa = \theta_2(\cW_0^i)$}
\EndIf
\EndFor
\For{$i_1 = 1,\dots,\ell-1$}
\For{$i_2  = i_1 + 1,\dots,\ell$}
\State{$\kappa = \kappa +1$}
\State{$w^\kappa = \max_{x \in \cB^{i_1}_0,\, y \in \cB^{i_2}_0} 
                   \theta_2(\mathrm{span}(x,y))$}
\EndFor
\EndFor     
\State{$\hat \theta_2 = \max_{i=1,\dots,\kappa} w^i$.}
\end{algorithmic}
\end{minipage}
\end{center}
\medskip

Note that the algorithm avoids to distinguish cases. If
$\dim(\cW_0^{i_1}) = \dim(\cW_0^{i_2}) = 1$ then 
$\theta_2(\cW_0^{i_1}\oplus\cW_0^{i_2})$ is computed for a single
two-dimensional subspace.
In the case $\dim(\cW_0^{i_1}) + \dim(\cW_0^{i_2}) = 3$, we solve
a one-dimensional optimization problem with the tools, described in
the case $s=1$. If $\dim(\cW_0^{i_1}) + \dim(\cW_0^{i_2}) = 4$, then the
optimization problem is two-dimensional, and we apply the
MATLAB-command \texttt{fminsearch}, which uses a
derivative-free method for finding minima of unconstrained multivariable
functions.

In all cases, we avoid  numerical errors during the iteration of $\Phi(j,0)x$ for
$x\in\cW_0^i$
(i.e.\ convergence towards the
most unstable direction) by
renormalizing the resulting output to $\cW_{j}^i$ after 
each step.

\subsection{Numerical experiments}
\label{sec5.2}
We apply our algorithm from Section \ref{algorithm} to several
models. First we reconsider some autonomous difference equations
from cf.\ \cite{BeFrHu20}. For this class of systems the
algorithm from \cite[Section 6]{BeFrHu20}
uses a series of Schur decompositions and  one-dimensional
optimization if necessary. Although this is more efficient for autonomous
systems, we still apply in the following our general algorithm
to the autonomous case in order to illustrate its performance.
Furthermore, we apply both algorithms in Section \ref{auto} to
autonomous systems and compare the results.

\subsubsection{Two-dimensional models}
\label{sec5.2.1}
We begin with several two-dimensional models 
for which angular values are analytically known. For these examples 
we  always find point spectrum 
which we approximate by upper and lower Bohl exponents. In some
cases upper and lower exponents coincide up to machine accuracy, while
in other models, we numerically observe intervals of length $\approx 10^{-3}$.
In the following we denote by
$T_\varphi = \left(\begin{smallmatrix} \cos \varphi &
      -\sin\varphi\\\sin\varphi & \cos\varphi\end{smallmatrix}\right)$
a rotation matrix.

\begin{table}[hbt]
\begin{center}
\begin{tabular}{c|c|c|c}
  $A_n$ & spectral intervals & $\hat \theta_{1,\mathrm{num}}$ &
  $|\hat \theta_1- \hat \theta_{1,\mathrm{num}}|$\\\hline & & \\[-4mm] 
$\left(\begin{smallmatrix}
    2 & 0\\ 0 & 3
  \end{smallmatrix}\right)$ &$\begin{array}{l}\cI_1 = [3,3]\\
  \cI_2 =[2,2]\end{array}$  & $5\cdot 10^{-15}$ & $5\cdot 10^{-15}$
  \\[1mm]\hline && \\[-4mm]   
$\left(\begin{smallmatrix}
    \cos \varphi & \sin \varphi \\ \sin \varphi & -\cos \varphi
  \end{smallmatrix}\right)$ & $\cI_1 = [1,1]$
  &$\frac \pi 2 \scriptstyle{- 3\cdot 10^{-5}} $ & $3\cdot 10^{-5}$
  \\[1mm]\hline\hline && \\[-4mm]  
$T_{(n+1)\varphi}\cdot
  \left(\begin{smallmatrix}
    2 & 0\\ 0 & 3
  \end{smallmatrix}\right)
                \cdot T_{-n \cdot\varphi}$&$\begin{array}{l}\cI_1 =
   [2.996, 3.000] \\ \cI_2=[2.000, 2.002]\end{array}$ & $\tfrac 13
  \scriptstyle{+ 6\cdot 10^{-13}}$ & $6\cdot 10^{-13}$ \\[1mm]\hline && \\[-4mm] 
$ T_{n \cdot\varphi}$&$\cI_2=[1, 1]$ & $\tfrac \pi 4
  \scriptstyle{+3.1\cdot 10^{-4}}$   & $3.1\cdot 10^{-4}$
\end{tabular}
\end{center}
\caption{Spectral intervals and the first angular value for four
  examples. 
  We set  $\varphi = \tfrac 13$ and use our algorithm
with  $M=2000$ iterates. \label{2Dnuma}}
\end{table}

The models from the first two rows in Table \ref{2Dnuma} 
are autonomous and we obtain approximately the expected results, see
\cite{BeFrHu20}. 
The second example is a reflection which  exhibits the angular value
$\hat \theta_1 = \frac \pi 2$  with a somewhat smaller error.
The third model is constructed via a nonautonomous similarity transformation
with rotation matrices, and we obtain the angular value $\hat \theta_1 =
\varphi = \frac 13$ with high accuracy.
Finally, in the last row of Table \ref{2Dnuma} we consider a rotation by the
angle $\varphi=\frac 13$ which is an irrational multiple of $\pi$.
The angle $\ang(u,T_{n\cdot\varphi} u)$ is
$\frac \pi 4$ on average, in agreement with our numerical
experiment. 

\subsubsection{Two autonomous examples}\label{auto}
Next we apply our algorithm to autonomous examples and compare
with the output of \cite[Algorithm 6.2]{BeFrHu20} based
on Schur decompositions. For this task we take the normal
form
\begin{align}\label{normalform}
     A(\rho,\varphi) =\begin{pmatrix} \cos(\varphi) & - \rho^{-1} \sin(\varphi)\\
    \rho \sin(\varphi) & \cos(\varphi) \end{pmatrix}, \quad
     0<\rho \le 1, \quad 0 <  \varphi \le \frac \pi 2
\end{align}
and consider first the matrix $A(\frac 17, \frac 13)$. The autonomous
algorithm uses an in-depth analysis of the first angular value of
\eqref{normalform}, given in 
\cite[Theorem 6.1]{BeFrHu20}. The resulting angular value is $\hat
\theta_{1,\mathrm{auto}} = 0.32106$.  
In coincidence with this result, the algorithm from Section \ref{algorithm} yields
the spectral interval $\cI_1 = [0.9991, 1.0009]$ and the angular value  $\hat
\theta_{1,\mathrm{num}} =0.32175$. 

Next, we analyze the four-dimensional matrix
$$
A = 
\begin{pmatrix}
A(1,\frac 12) & I_2\\ 0 & \eta A(\frac 12, 1.4)
\end{pmatrix}
$$
with $\eta = 1.2$. This example is crucial since the angular
  values cannot be read off from the diagonal blocks only. Rather, one has
  to compute orthogonal bases of both two-dimensional invariant subspaces.
  The analysis and the corresponding algorithm in \cite[Section 6.3.2]{BeFrHu20}
 provide the angular value $\hat \theta_{1,\mathrm{auto}} = 1.355003$.
Our current algorithm from Section \ref{algorithm} findes the spectral
intervals 
$$
\cI_1 = [1.1992, 1.2008]\quad \text{and}\quad \cI_2 = [1.0000, 1.0000]
$$
with corresponding angular values
$$
\hat \theta_{1,\mathrm{num}} = \theta_1(\cW_0^1) = 1.355095\quad
\text{and}\quad \theta_1(\cW_0^2) = 0.500000. 
$$
This fits well to the results of the autonomous algorithm. 

\subsubsection{Angular values and tangent spaces}\label{tspace}
For a geometric interpretation of angular values, we consider an
invertible discrete-time dynamical system defined on $\Z$.
Let $F_n :\R^d\to \R^d$, $n\in\Z$ be a family of
$\cC^2$- diffeomorphisms and let
\begin{equation}\label{nonlin}
x_{n+1} = F_n(x_n),\quad n \in \Z.
\end{equation}
Denote by $\Psi$ the solution operator of \eqref{nonlin}. 
For a bounded trajectory $\xi_\Z:= (\xi_n)_{n\in\Z}$, we introduce
the corresponding variational equation 
\begin{equation}\label{varia}
u_{n+1} = DF_n(\xi_n)u_n,\quad n\in\Z,
\end{equation}
with solution operator $\Phi$.   
Note that \eqref{nonlin} has the form \eqref{diffeq} with $A_n =
DF_n(\xi_n)$.
We further assume that the bounded trajectory $\xi_\Z$ is
hyperbolic, i.e.\ $1$ is in the resolvent set of the dichotomy
spectrum of \eqref{varia}. 

Stable and unstable fiber bundles of $\xi_\Z$ are defined at time
$k\in\Z$ as
\begin{align*}
\cF_k^s & := \left\{x\in\R^d: \lim_{n\to \infty} |\Psi(n,k)(x) -
             \xi_n|=0\right\},\\
\cF_k^u & := \left\{x\in\R^d: \lim_{n\to -\infty} |\Psi(n,k)(x) -
             \xi_n|=0\right\},
\end{align*}
and we denote corresponding tangent spaces by
$T_{\xi_k} \cF_k^{s}$ and $T_{\xi_k} \cF_k^{u}$. These tangent spaces are related to 
spectral bundles from Section \ref{SackerSell} as follows.
Fix $i\in \{1,\dots,\ell\}$ such that $1\in R_i$, where $R_i$ denotes
the $i$-th resolvent interval. We conclude from \cite[Theorem
4.6.4]{p10} and \eqref{dichbundle} that
\begin{equation*}
T_{\xi_k} \cF_k^{s} = \range(P_{k,i}^s) = \bigoplus_{j=i}^\ell \cW_k^j
\quad
\text{and}\quad
T_{\xi_k} \cF_k^{u} = \range(P_{k,i}^u) = \bigoplus_{j=1}^{i-1} \cW_k^j.
\end{equation*}

For two-dimensional systems with $\dim(\cF_k^s) = \dim(\cF_k^u)
= 1$, we observe that 
\begin{align*}
\theta_1(\cW_0^{2}) &= \frac 1M \sum_{j=1}^{M}
  \ang(\Phi(j-1,0)\cW_0^{2} , \Phi(j,0)\cW_0^{2})
= \frac 1M \sum_{j=1}^{M}
  \ang(T_{\xi_{j-1}}\cF_{j-1}^{s}, T_{\xi_{j}}\cF_{j}^{s}),\\
\theta_1(\cW_0^{1}) &= \frac 1M \sum_{j=1}^{M}
  \ang(\Phi(j-1,0)\cW_0^{1} , \Phi(j,0)\cW_0^{1})
= \frac 1M \sum_{j=1}^{M}
  \ang(T_{\xi_{j-1}}\cF_{j-1}^{u}, T_{\xi_{j}}\cF_{j}^{u})
\end{align*}
describes the angle between successive stable resp.\ unstable tangent
spaces on average. 
The maximum of these two averages is $\hat \theta_1 =
\max\{\theta_1(\cW_0^1),\theta_1(\cW_0^2)\}$.  

In higher dimensional systems, a geometric interpretation of angular
values is more involved. If a three-dimensional model, 
for example, satisfies $\dim(\cF_k^u) = 1$ and $\dim(\cF_k^s) = 2$, we get
for the one-dimensional unstable subspace
$$
\theta_1(\cW_0^1) = \frac 1M \sum_{j=1}^{M}
\ang(T_{\xi_{j-1}}\cF_{j-1}^u, T_{\xi_{j}}\cF_{j}^u).
$$

Next, we consider the two-dimensional stable subspace 
$$
\cW_0^s := 
\begin{cases} 
\cW_0^2, & \text{if } \dim(\cW_0^2) = 2,\\
\cW_0^2 \oplus \cW_0^3,& \text{otherwise.}
\end{cases}
$$
The first angular value 
$$
\theta_1(\cW_0^s) = \sup_{v\in T_{\xi_0} \cF_0^s} \frac 1M
\sum_{j=1}^{M} \ang(\Phi(j-1,0)v,\Phi(j,0)v)
$$
describes on average the maximal angle between successive one-dimensional
subspaces in $T_{\xi_\Z}\cF_\Z^s$. 
Combining these result gives $\hat \theta_1 =
\max\{\theta_1(\cW_0^1),\theta_1(\cW_0^s)\}$. 

For three-dimensional models, also second angular values are of
interest. 
The average angle between successive two-dimensional stable subspaces
is given by 
$$
\theta_2(\cW_0^s) = \frac 1M
\sum_{j=1}^{M}\ang(T_{\xi_{j-1}}\cF_{j-1}^s,
T_{\xi_{j}}\cF_{j}^s)
$$
and the latter formula provides a nice geometrical interpretation.    
However, for computing $\hat \theta_2$, we have to consider further
subspaces:
$$
\hat \theta_2 = \max(\{\theta_2(\cW_0^s)\}\cup \{\theta_2(V): V=
\cW_0^1\oplus \Span(u): u \in \cW_0^s\}).
$$
As a consequence, second and higher angular values are in general not
achieved within the stable resp.\ unstable tangent bundle. 

\subsubsection{Models of H\'enon type}\label{sec_hen}
We illustrate the geometric interpretation of angular values from
Section \ref{tspace} with two autonomous, nonlinear systems. 
Of interest are the two-dimensional H\'enon map \cite{h76} as well as
its three dimensional variant 
$$
F^2\begin{pmatrix}x_1\\x_2\end{pmatrix} =
  \begin{pmatrix}
  1+x_2-1.4x_1^2\\0.3x_1  
\end{pmatrix},\quad
 F^3\begin{pmatrix}x_1\\x_2\\x_3\end{pmatrix} =
 \begin{pmatrix}
  1+ x_3 - 1.4 x_1^2\\x_1+x_3\\0.2 x_1 + 0.1 x_2 
 \end{pmatrix}.
 $$
The latter model is constructed similar to
\cite[Example 2]{bk97}, and possesses, like the original  H\'enon map,
a non-trivial attractor. 
 
We choose $M=2000$ and compute angular values for the
corresponding variational equation \eqref{varia}. 
Note that we apply the algorithm to these models
  even though we do not know whether the hyperbolicity assumptions are
  satisfied. 

\subsubsection*{The two-dimensional H\'enon model}
We choose the initial point close to the
H\'enon attractor $\xi_{-50} = \begin{pmatrix} 0.7555 &
  0.1671 \end{pmatrix}^\top$ and obtain the spectral
intervals 
$$
\cI_1= [1.491, 1.549]\quad \text{and}\quad \cI_2 = [0.194, 0.201]
$$ 
with corresponding angular values
$$
\theta_1(\cW_0^1) = 0.3629,\quad
\theta_1(\cW_0^2) = 0.7506.
$$
The maximum $\hat \theta_1 = \theta_1(\cW_0^2)$ is achieved by the angle between
successive stable tangent spaces. Stable and unstable
tangent spaces are shown in Figure \ref{TangHen2}. In addition, we
present approximations of the stable and of the unstable manifold of the fixed point $\xi$. 

\begin{figure}[hbt]
 \begin{center}
   \includegraphics[width=0.75\textwidth]{ 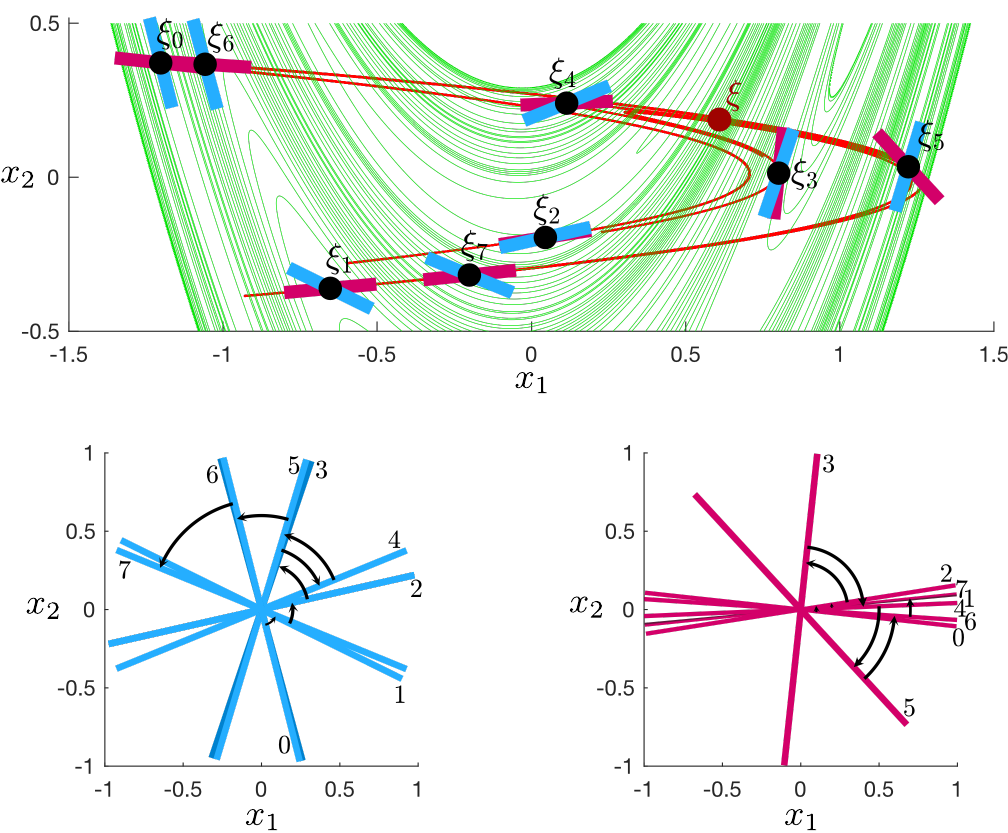}
 \end{center}
 \caption{\label{TangHen2} Upper panel: Stable (green) and unstable
   (red) manifolds of the fixed point $\xi$ of the two-dimensional H\'enon
   model. Lower panel: Successive stable (left) and 
   unstable (right) tangent spaces.}
\end{figure}

 Figure \ref{H2vary} shows the  dependence of $\hat\theta_{1,M}$ from \eqref{thetafinal} on the
 length $M$ of the finite interval. We observe the typical slow convergence of
the ergodic average \eqref{winkelrechnung} as $M\to \infty$. 
\begin{figure}[hbt]
 \begin{center}
   \includegraphics[width=0.60\textwidth]{ 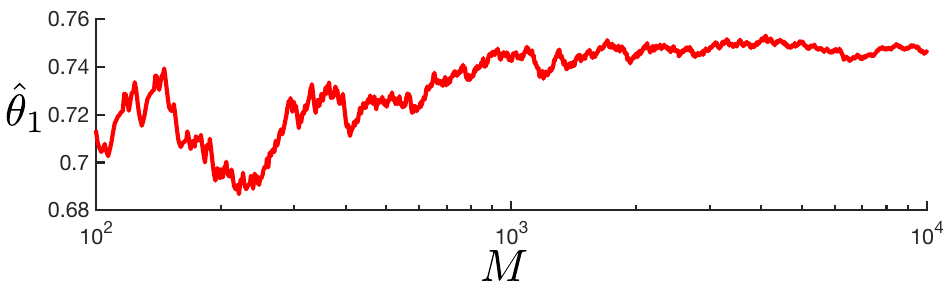}
 \end{center}
 \caption{\label{H2vary}Angular value $\hat \theta_1$ in the
   two-dimensional H\'enon model, computed for $M\in
   [10^2,10^4]\cap \N$.}
\end{figure}
So far the approximate value $\hat \theta_{1,M}$ in
\eqref{thetafinal} was computed by starting at time $k=0$. 
It is instructive to compare these results with the values obtained
from 
\begin{equation}\label{angk}
\theta_1(k,M) := \sup_{V\in\cG(1,2)} \frac 1M
\sum_{j=k+1}^{k+M}\ang(\Phi(j-1,k)V, \Phi(j,k)V), 
\end{equation}
where the computation starts at a later time $k\in\N$. 
Taking the supremum over $k$ results in an 
approximation of uniform outer angular values $\hat \theta_{[s]}$ from
Definition \ref{defangularvalues}. Here we analyze the dependence of
\eqref{angk} 
on $k\in\{0,\dots,10^4\}$ for the values $M\in\{10^2, 10^3, 10^4\}$. 
Figure \ref{Hen2k} shows the result for 
the two-dimensional H\'enon model.
\begin{figure}[hbt]
 \begin{center}
   \includegraphics[width=0.70\textwidth]{ 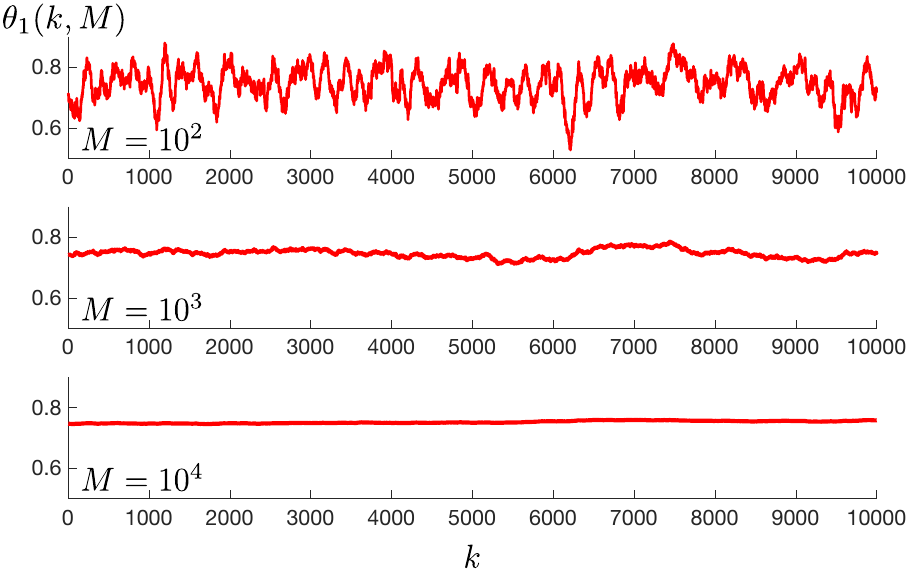}
 \end{center}
 \caption{\label{Hen2k} Computation of $\theta_1(k,M)$ for the
   two-dimensional H\'enon map.}
\end{figure}
The numerical data suggest convergence as $M\to\infty$ uniformly in
$k$. However, we argue that we do not expect this to hold in theory for the following reason.
Assume the orbit $(\xi_n)_{n\in\Z}$ is dense in the H\'enon
attractor, which contains the saddle fixed point and its unstable manifold.
Then there exist arbitrarily long time intervals on which
the orbit is close to the local unstable manifold. Thus, one finds 
arbitrarily large values of $k$ for which  $\theta_1(k,M)$ is close to $0$.
It will be extremely difficult to observe this
non-uniformity w.r.t.\ $k$ numerically since one needs $k$-values
which grow exponentially with the observation length $M$.

In the next experiment we show that a 
naive approach to  compute the angular
values by forward iteration tends to fail. We consider
\[
\theta_1(V_\varphi) \quad \text{with}\quad 
V_\varphi =
\mathrm{span}\begin{pmatrix}\cos(\varphi)\\\sin(\varphi)\end{pmatrix} 
\]
for $10^5$ equidistant values of $\varphi \in [0,\pi]$ and
display the values $\theta_1(V_{\varphi})$ from \eqref{winkelrechnung}
with $M=2000$ fixed, in Figure \ref{naiv}.
\begin{figure}[hbt]
 \begin{center}
   \includegraphics[width=0.80\textwidth]{ 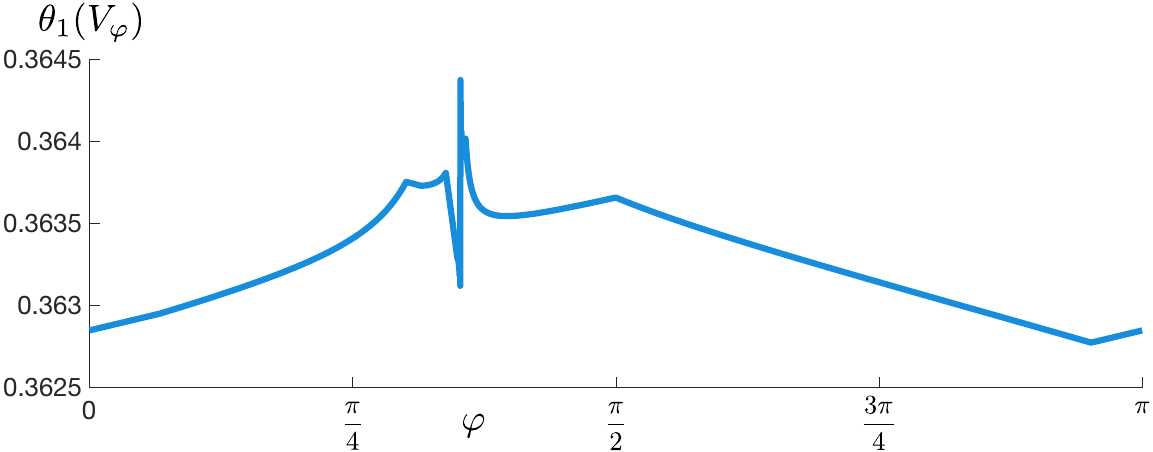}
 \end{center}
 \caption{\label{naiv}Computation of
   $\theta_1(V_\varphi)$ for the two-dimensional H\'enon map for
   $10^5$ values of $\varphi\in[0,\pi]$.} 
\end{figure}
As we know, the angular value $\hat \theta_1 = 0.7506$ is achieved
by the tangent space to the stable fiber.
However, each subspace $V_\varphi$ with $\varphi$ chosen from a very
fine grid, is pushed during
$M=2000$ iterations towards the unstable subspace. The peak at
$\varphi \approx 1.1$ indicates the appropriate stable
subspace, but the observed maximum value is only slightly larger than
$\theta_1(\cW_0^1) = 0.3629 \ll  0.7506= \hat \theta_1 $.

\subsubsection*{The three-dimensional H\'enon model}
We choose $\xi_{-81}
= \begin{pmatrix} 0.2 & 0.1 & 0\end{pmatrix}^\top$ as initial point
and obtain the spectral
intervals 
$$
\cI_1 = [1.406, 1.442],\quad \cI_2 = [0.378,
0.387]\quad  \text{and}\quad \cI_3 = [0.318, 0.333].
$$
The corresponding unstable fibers $\cW_k^{1}=T_{\xi_k}\cF_k^u$ are
one-dimensional and 
the direct sum of the stable fibers $\cW_k^2 \oplus \cW_k^3 =T_{\xi_k}\cF_k^s$ is
two-dimensional. These subspaces are 
shown in Figure \ref{TangHen3} for $0\le k\le 4$.

The angular value $\hat \theta_1 = 0.8129$ is achieved in $\cW_0^1$:
$$
\theta_1(\cW_0^1) = 0.825,\quad \theta_1(\cW_0^2) = 0.693,\quad \theta_1(\cW_0^3) = 0.757.
$$
While the first angular value $\hat \theta_1$ describes the angle
between successive unstable tangent spaces $T_{\xi_k} \cF_k^u$ on
average, the average angle between successive stable tangent spaces $T_{\xi_k}
\cF_k^s$ is given by the second angular value $\hat \theta_2 = 0.839$, see
Figure \ref{TangHen3}. This value
  is found as the maximum of the following numerical computations
$$
\theta_2(\cW_0^2\oplus \cW_0^3) = 0.839,\quad
\theta_2(\cW_0^1\oplus \cW_0^2) = 0.611,\quad
\theta_2(\cW_0^1\oplus \cW_0^3) = 0.703.
$$
Note that generally, angular values are not achieved within stable
respectively unstable subspaces. The
three-dimensional H\'enon model seems to be exceptional in this regard.
In general, invariant
subspaces (in a nonautonomous sense) in which angular values are
achieved, are not characterized by contracting or expanding dynamics;
see Section \ref{tspace} and \ref{3Dparam}. 

\begin{figure}[hbt]
 \begin{center}
   \includegraphics[width=0.75\textwidth]{ 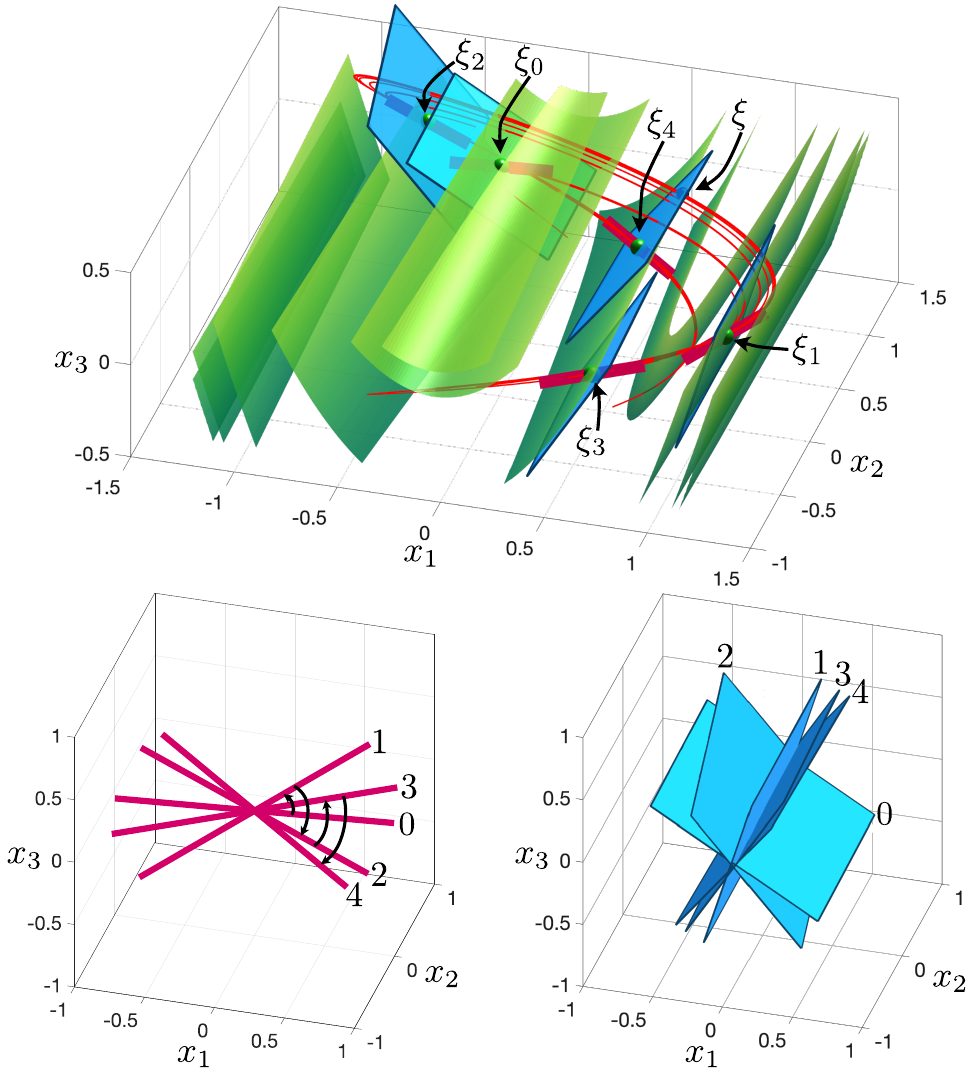}
 \end{center}
 \caption{\label{TangHen3} Upper panel: Stable (green) and unstable
   (red) manifolds 
   of the fixed point $\xi$ of the three-dimensional H\'enon model. 
  Lower panel: Successive one-dimensional unstable (left) and
   two-dimensional stable (right) tangent spaces.}
\end{figure}

\subsubsection{A three-dimensional toy model}
\label{3Dparam}
We construct a three-dimensional 
  nonautonomous model which depends on a parameter but has constant
  angular values.  However, the parameter changes the spectrum and
  the stability properties as well as the
  spectral bundles.   
   We show that the algorithm succeeds in finding the correct angular
    value regardless whether or not the optimization step is invoked.

With $\varphi = \tfrac 13$ and
parameter $\lambda>0$ the model has the form
 \begin{equation}\label{parmod}
  A_n = 
  T^{1,2}_{\varphi(n+52)}\cdot
  \begin{pmatrix}
    \frac 12 &0 &0\\ 0 &\lambda &0\\ 0 &0 &3
  \end{pmatrix}
  \cdot T^{1,2}_{-\varphi (n+51)}
  \text{ with }
  T^{1,2}_{\varphi}=
  \begin{pmatrix}
    \cos \varphi & -\sin\varphi & 0\\
    \sin \varphi & \cos \varphi & 0\\
    0 & 0 & 1
  \end{pmatrix}
  \end{equation}
and point spectrum $\{\frac 12, \lambda, 3\}$. One can show that both
angular values are independent of the parameter $\lambda$ and satisfy
$\hat \theta_1 = \hat \theta_2 = \varphi = \frac 13$. By 
choosing $\lambda=\frac 12$ and $\lambda=2$ we vary the
dimensions of the corresponding spectral bundles and study the
consequences. 

\subsubsection*{$\lambda=\tfrac 12$} The numerical  spectral intervals are
  $\cI_1=[3,3]$ and $\cI_2 = [\frac 12,\frac 12]$, and we obtain the first angular
  value
  $$
  \theta_1(\cW^1_0) = 0,\quad \max_{x\in\cB_0^2}\theta_1(x)
  = \frac 13, \quad \hat\theta_1 = \frac 13, 
  $$
 and the second angular value
  $$
   \theta_2(\cW^2_0) = 0,\quad \max_{x\in\cB_0^2}
   \theta_2( \cW_0^1\oplus\mathrm{span}(x))= \frac{1}{3},\quad
   \hat \theta_2 = \frac 13,
   $$
   cf.\ Figure \ref{P05}. Note that the subspaces, in which the
   maximum is achieved are not unique. The component in $\cW_0^2$ can
   be chosen arbitrarily.
  \begin{figure}[hbt]
 \begin{center}
   \includegraphics[width=0.7\textwidth]{ 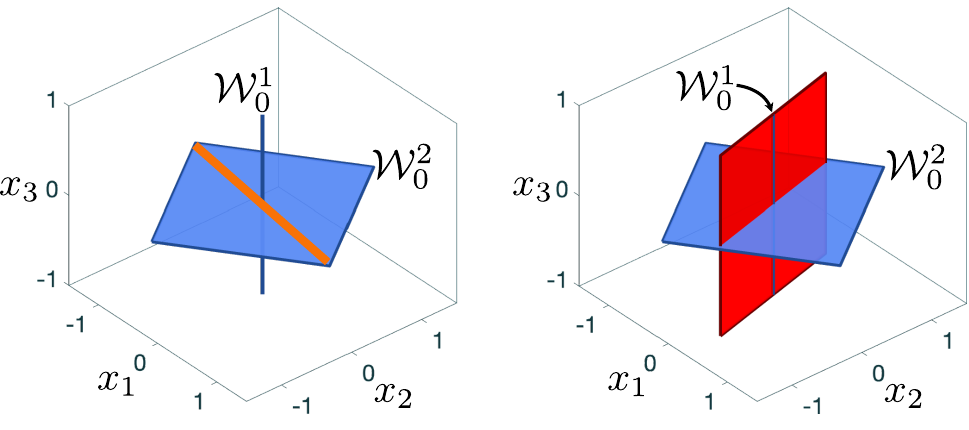}
 \end{center}
 \caption{\label{P05}Spectral bundles (blue) of \eqref{parmod} for
    $\lambda=\frac 12$. The angular values $\hat \theta_{1,2}$ are
   achieved at subspaces which are shown in red.}
\end{figure}
\subsubsection*{$\lambda= 2$} In this example, we obtain three  spectral intervals
  $\cI_1=[3,3]$, $\cI_2 = $\linebreak $ [1.9955, 2.0000]$ and $\cI_3 = [0.5000,
  0.5011]$.
  We find that 
  $$
  \theta_1(\cW^1_0) = 0,\quad  \theta_1(\cW^2_0) = \frac 13,\quad
  \theta_1(\cW^3_0) = \frac 13.
  $$
  Thus $\hat \theta_1 = \frac 13$, where the maximum is achieved for two
  fibers. In Figure \ref{P2}, the algorithm chooses $\cW_0^3$.
  The second angular value is also achieved for two different trace spaces:
  $$
  \theta_2(\cW^1_0\oplus \cW_0^2) = \frac 13 ,\quad
 \theta_2(\cW^1_0\oplus \cW_0^3) = \frac 13 ,\quad
 \theta_2(\cW^2_0\oplus \cW_0^3) = 0 ,\quad
 \hat \theta_2 = \frac 13.
 $$
 \begin{figure}[H]
 \begin{center}
   \includegraphics[width=0.7\textwidth]{ 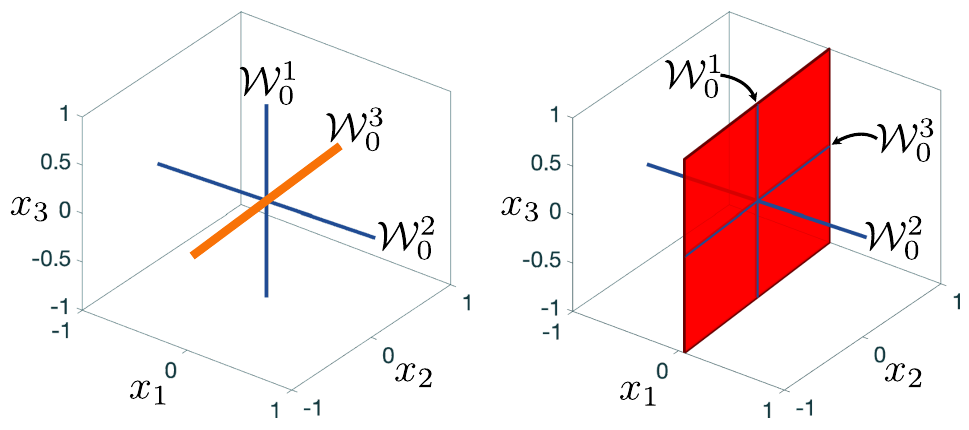}
 \end{center}
 \caption{\label{P2}Spectral bundles (blue) of \eqref{parmod} for
   $\lambda= 2$. The angular values $\hat \theta_{1,2}$ are
   achieved at subspaces which are shown in red. }
\end{figure}

  \subsubsection{A random dynamical system}
  \label{RDS}
For $\varphi = 0.2$, we define
  $$
  B_1 =
  \begin{pmatrix}
    2\cos(\varphi) & -2\sin(\varphi) & 0\\
    2 \sin(\varphi) & 2 \cos(\varphi) & 0\\
    0 & 0 & 3
\end{pmatrix},\quad
B_2 =
  \begin{pmatrix}
  1 & 0 & 0\\0 & 1 & 0\\ 0 & 0 & 5
\end{pmatrix}
$$
and construct the 3-dimensional random dynamical system
\begin{equation*}\label{rand}
A_n = B_r,\quad \text{ where } r\in\{1,2\} \text{ is uniformly
  distributed for each } n \in \I.
\end{equation*}
This random dynamical system allows an explicit study of the
dichotomy spectrum, see \cite[Remark 4.2.9]{A1998}, and of angular values.
One has $\Sigma_{\mathrm{ED}} = \{\lambda_1, \lambda_2\}$ with
$$
\lambda_1 = \sqrt{3\cdot 5} \approx 3.872,\quad \lambda_2 =
\sqrt{1\cdot 2} \approx 1.414
$$
and the corresponding fiber bundles are 
$$
\cW_0^1 = \Span\left(\begin{smallmatrix}0\\0\\1\end{smallmatrix}\right),\quad
\cW_0^2 =
\Span\left(\left(\begin{smallmatrix}1\\0\\0\end{smallmatrix}\right), 
\left(\begin{smallmatrix}0\\1\\0\end{smallmatrix}\right)\right).
$$
The first and second angular values are equal to $1$, since
$$
\theta_1(\cW_0^1) = 0,\quad
\forall x\in\cB_0^2: \theta_1(x) = \frac {0.2+0}2 = 0.1
$$
and
$$
\theta_2(\cW_0^2) = 0,\quad
\forall x\in\cB_0^2: \theta_2(\cW_0^1\oplus \Span(x)) = 0.1.
$$

These analytic results are in coincidence with the output of our
numerical algorithm.
One realization gives the spectral intervals
$\cI_1 = [3.86, 3.92]$ and $\cI_2 =[1.39, 1.42]$ and the first angular
value
$$
\theta_1(\cW^1_0) = 0,\quad  \max_{x\in\cB_0^2}\theta_1(x) =
0.099631 = \hat \theta_1.
$$
A numerical computation of the second angular value yields 
$$
\theta_2(\cW^2_0) = 0,\quad  \max_{x\in\cB_0^2}\theta_2(\cW_0^1\oplus
\mathrm{span}(x)) = 0.099631 =  \hat \theta_2.
$$

\subsubsection{Coupled oscillators}
\label{coupledoscillator}
  We consider a canonical model for two nonlinear oscillators  with a linear
  diffusion-like coupling originating from \cite{ado87}. It has been frequently used as
  a model problem for analyzing and computing invariant tori \cite{dlr91, dl95,r2000}:
\begin{equation}\label{ode}
x' = G(x),\quad
  G(x) =
  \begin{pmatrix}
 x_1 + p_1 x_2 - ( x_1^2 + x_2^2 ) x_1 - \lambda (x_1 + x_2 - x_3 - x_4 ) \\
 - p_1 x_1 +  x_2 - ( x_1^2 + x_2^2 ) x_2 - \lambda (x_1 + x_2 - x_3  -x_4 ) \\
 x_3 + p_2 x_4 - ( x_3^2 + x_4^2 ) x_3 + \lambda (x_1 + x_2 - x_3 - x_4 ) \\
 - p_2 x_3 + x_4 - ( x_3^2 + x_4^2 ) x_4 + \lambda (x_1 + x_2 - x_3 - x_4 ) 
\end{pmatrix}.
\end{equation}
For parameter values $\lambda$ in some interval $[0,\lambda_{\mathrm{crit}})$
  the model possesses an invariant torus which breaks down at a critical
  value $\lambda_{\mathrm{crit}}>0$; see \cite{dl97}, \cite{r2000}
  for a thorough analysis of this phenomenon and numerical results
  for  $p_1 = p_2 = 0.55$ where $\lambda_{\mathrm{crit}}\approx 0.2607$. It is worth noting that this system has
  a $\Z_2$-symmetry which can be written as $G(Sx)=SG(x), x\in \R^4$ for
  the permutation $S(x_1,x_2,x_3,x_4)^\top=(x_3,x_4,x_1,x_2)^\top$.
In the following we consider the $1$-flow
$F:\R^4\to \R^4$ of \eqref{ode} which inherits the symmetry $S\circ F=F\circ S$.
In particular, the symmetric space $X_s=\{x\in \R^4:Sx=x\}$ and the antisymmetric space $X_a=\{x\in \R^4:Sx=-x\}$ are invariant. As in
Section \ref{tspace} we determine angular values of the variational
equation along an $F$-orbit. The $1$-flow is approximated by  
the explicit Euler scheme which respects the symmetry.
The following data are obtained with step size $h=0.01$, and
orbit length $M=1000$.
For $21$ equidistant values of $\lambda \in [0.1, 0.3]$, we compute
the first and second angular value in Figure \ref{la0103}.
We choose for each value of $\lambda$ a random initial point in $[-1,1]^4$.  
\begin{figure}[hbt]
 \begin{center}
   \includegraphics[width=0.7\textwidth]{ 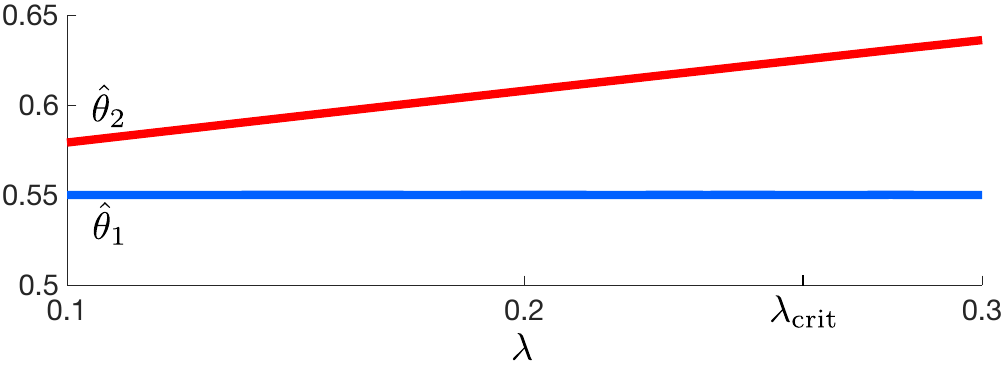}
 \end{center}
 \caption{\label{la0103} First (blue) and second angular value (red) for
   $\lambda\in [0.1,0.3]$.}
\end{figure}

 The breakdown of the invariant torus can be understood by comparing Lyapunov-type
numbers which measure the contraction within the torus and toward the
torus; see \cite{dl97}. Angular values are not suitable for this task.
They measure the maximal average rotation of one- and
two-dimensional subspaces in an attractor of the system. In the
coupled oscillator system almost all trajectories converge to the
in-phase symmetric orbit 
(lying on the torus if it exists). Thus 
the spectral bundle consists of the Floquet spaces (including the flow
direction) and the angular 
values represent the maximal rotation of one- or two dimensional
subspaces composed of Floquet spaces.
In case of an antisymmetric initial value (and since the
computation preserves antisymmetry) the orbit converges towards
the unstable antisymmetric periodic orbit (lying on the invariant torus
if it exists). Then angular values measure the maximum rotation of
Floquet subspaces belonging to this unstable orbit. For both initial
data the corresponding angular value passes smoothly through the
parameter domain where the torus breaks down.
Close to the breakdown of the invariant torus, we compute two orbits for
$\lambda =0.26$ with antisymmetric initial value $x_{a}=(0.1, 1, -0.1,
-1)^\top$ and nonsymmetric initial value $x_{r}=(0.1, 1, 0.2,
0.1)^\top$ at time $-50$,
shown in Figure
\ref{orb026}. Table \ref{tab1} and \ref{tab2} contain the corresponding
spectral intervals and angular values. Note that one spectral interval
always contains the trivial Floquet multiplier $1$ and that the symmetric
periodic orbit is orbitally stable while the antisymmetric one is unstable.
The findings are in accordance with the results in \cite{r2000}. 
The corresponding (Floquet-) subspaces at
time $50$ are depicted in Figure \ref{UR_Rot}. It is worth noting
that the $2$D fibers $\cW_0^1 \oplus \cW_0^3$ and $\cW_0^2 \oplus \cW_0^4$ 
agree with the symmetric resp. the antisymmetric space. Their invariance
ensures that the second angular value is zero as confirmed by the
corresponding entries in Table \ref{tab2}.
\begin{figure}[hbt]
 \begin{center}
   \includegraphics[width=0.5\textwidth]{ 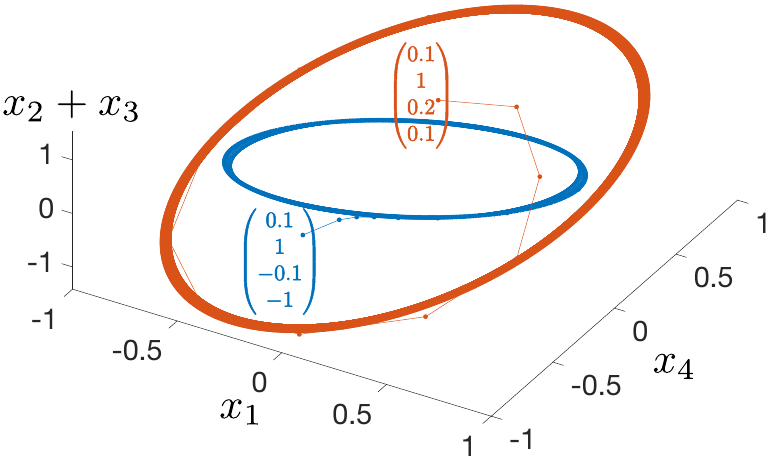}
 \end{center}
 \caption{\label{orb026} Two orbits for $\lambda =
0.26$ with antisymmetric and arbitrary initial value.}
\end{figure}

\begin{table}
\begin{center}
\begin{tabular}{c|c|c||c|c}
initial data  & \multicolumn{2}{c}{antisymmetric $x_a$} &
                \multicolumn{2}{c}{nonsymmetric $x_r$}\\
$1$D trace space & interval & $\theta_1$  & interval & $\theta_1$ \\\hline 
  $\cW_0^1$ & $[1.186, 1.195]$ &$0.194$ &$[0.999, 1.000]$&$0.550$\\[1mm]
  $\cW_0^2$ & $[0.995, 1.005]$ &$0.179$&$[0.621, 0.624]$&$0.549$\\[1mm]
  $\cW_0^3$ &$[0.902, 0.918]$ &$0.213$&$[0.132, 0.133]$&$0.550$\\[1mm]
  $\cW_0^4$ &$[0.379, 0.383]$ &$0.179$ &$[0.074, 0.0749]$&$0.550$\\
\end{tabular}
\end{center}
\caption{\label{tab1}First angular values for the trace spaces ($\approx$ $1$D Floquet spaces) of the orbits from Figure \ref{orb026} converging to
  the antisymmetric resp. the symmetric periodic orbit.
The maximum is $\hat{\theta}_1=0.213$ resp. $\hat{\theta}_1=0.550$.}
\end{table}

\begin{table}
\begin{center}
\begin{tabular}{c|c|c}
 initial data & antisymmetric $x_a$ & nonsymmetric $x_r$\\
 $2$D trace space & $\theta_2$  & $\theta_2$ \\\hline 
  $\cW_0^1\ \oplus \cW_0^2$ & $0.301$ & $0.625$\\[1mm]
  $\cW_0^1\ \oplus \cW_0^3$ &$0.000$ &$0.000$\\[1mm]
  $\cW_0^1\ \oplus \cW_0^4$ &$0.256$ & $0.625$\\[1mm]
  $\cW_0^2\ \oplus \cW_0^3$ &$0.311$ & $0.625$\\[1mm]
  $\cW_0^2\ \oplus \cW_0^4$ &$0.000$ & $0.000$ \\[1mm]
  $\cW_0^3\ \oplus \cW_0^4$ &$0.258$ &$0.625$ \\[1mm]
\end{tabular}
\end{center}
\caption{\label{tab2}Second angular value for the trace spaces ($\approx$ $2$D Floquet spaces) of the orbits from Figure \ref{orb026} converging to the
  antisymmetric resp. the symmetric periodic orbit. The maximum is $\hat{\theta}_2= 0.311$ resp. $\hat{\theta}_2=0.625$}
\end{table}

\begin{figure}[H]
 \begin{center}
   \includegraphics[width=0.8\textwidth]{ 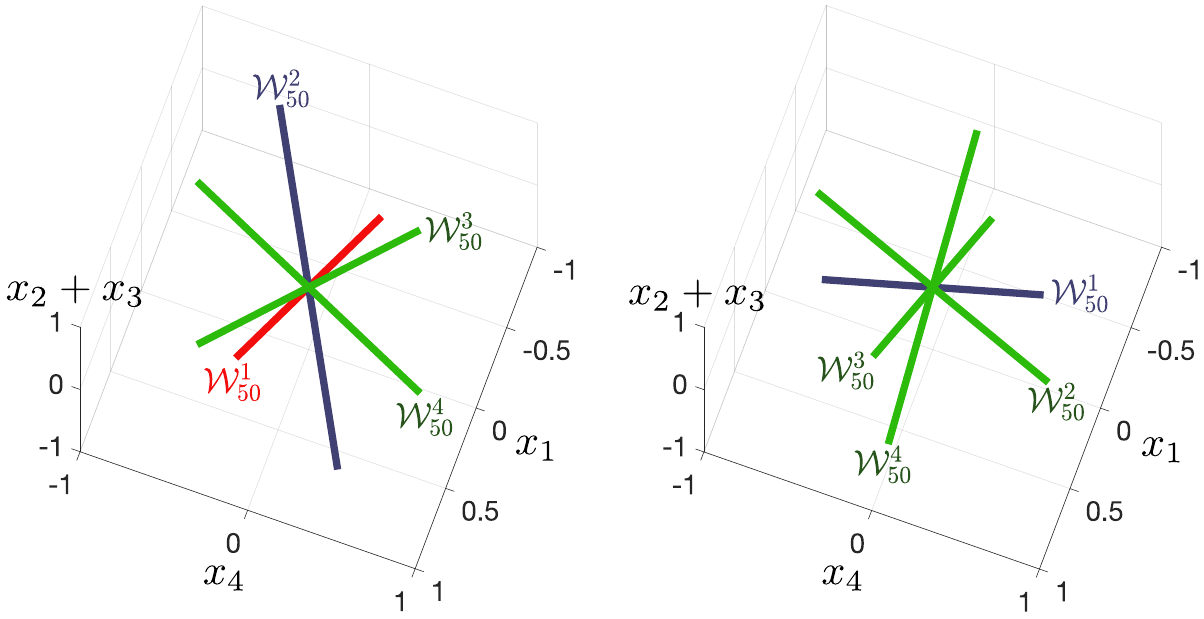}
 \end{center}
 \caption{\label{UR_Rot} Spectral bundle of $4$ subspaces ($\approx$ Floquet spaces) at time $50$ for
   $\lambda =0.26$ and antisymmetric (left) resp. nonsymmetric initial data
   (right). Coloring: green (stable), red (unstable), blue (neutral
   $\approx$ flow direction).}
\end{figure}

\section*{Acknowledgments}
The authors thank Gary Froyland for valuable comments which led to
improvements of this paper. Our thanks also go to an anonymous referee
  whose comments led to substantial clarifications, in particular in
  Section \ref{sec5}. 
Both authors are grateful to the Research Centre for Mathematical
Modelling ($\text{RCM}^2$) at Bielefeld University for continuous
support of their joint research. 
For further support,   WJB thanks the CRC 1283 "Taming Uncertainty and
Profiting from Randomness and Low Regularity in Analysis, Stochastics
and Their Applications" 
and TH thanks the Faculty of Mathematics at Bielefeld
University. 
 

\bibliographystyle{abbrv}

\end{document}